\documentclass[12pt]{article}
\usepackage{amsfonts}
\usepackage{amssymb, amscd, amsthm}
\usepackage{latexsym}
\usepackage{multirow}
\usepackage{amsmath}
\usepackage[all]{xy}
\usepackage[dvips]{graphicx}
\usepackage{mathrsfs}
\usepackage{fancyhdr}

\newtheorem{lemma}[subsection]{Lemma}
\newtheorem{thm}[subsection]{Theorem}
\newtheorem{prop}[subsection]{Proposition}
\newtheorem{rem}[subsection]{Remark}

\newtheorem{defn}[subsection]{Definition}

\newtheorem{cond}[subsection]{Condition}

\newcommand{\ra}{\rightarrow}
\newcommand{\mone}{\mo_{\p^2}(1)}
\newcommand{\mmon}{\mo_{\p^2}(-1)}
\newcommand{\mo}{\mathcal{O}}

\newcommand{\mg}{\mathcal{G}}
\newcommand{\ma}{\mathcal{A}}
\newcommand{\mb}{\mathcal{B}}

\newcommand{\ms}{\mathcal{S}}
\newcommand{\ts}{\mathtt{S}}
\newcommand{\mc}{\mathcal{C}}
\newcommand{\mr}{\mathcal{R}}
\newcommand{\z}{\theta}

\newcommand{\e}{\epsilon}
\newcommand{\ee}{\eta}

\newcommand{\p}{\mathbb{P}}
\newcommand{\bc}{\mathbb{C}}
\newcommand{\bl}{\mathbb{L}}
\newcommand{\km}{\mathfrak{m}}

\begin{document}
\fontsize{12pt}{14pt} \textwidth=14cm \textheight=21 cm
\numberwithin{equation}{section}
\title{Moduli spaces of semistable sheaves of dimension 1 on $\mathbb{P}^2$.}
\author{Yao YUAN}
\date{\small\textsc 
MSC, Tsinghua University, Beijing 100084, China
\\ yyuan@mail.math.tsinghua.edu.cn.}
\maketitle
\begin{flushleft}{\textbf{Abstract.}}
Let $M(d,\chi)$ be the moduli space of semistable sheaves of rank 0, Euler characteristic $\chi$ and first Chern class $dH~(d>0)$, with $H$ the hyperplane class in $\p^2$.  We give a description of $M(d,\chi)$,  viewing each sheaf as a class of matrices with entries in $\bigoplus_{i\geq0}H^0(\mo_{\p^2}(i))$.  We show that there is a big open subset of $M(d,1)$ isomorphic to a projective bundle over an open subset of a Hilbert scheme of points on $\p^2.$  Finally we compute the classes of $M(4,1)$, $M(5,1)$ and $M(5,2)$ in the Grothendieck ring of varieties,  especially we conclude that $M(5,1)$ and $M(5,2)$ are of the same class. 
\end{flushleft}

\section{Introduction.}
Moduli spaces $M$ of semistable sheaves of dimension 1 on surfaces are very interesting and many people have studied on them.  On $K3$ or abelian surfaces,  for a large number of $M$, Yoshioka has given explicitly the deformation classes of them in \cite{ky}.  Le Potier studied a lot on $M$ for $\p^2$ such as their Picard groups and rationalities in \cite{lee}. Dr\'ezet and Maican studied sheaves of dimension 1 on $\p^2$ with multiplicity 4,5 and 6, via their locally free resolutions (see \cite{jmdmm},\cite{mmo} and \cite{mmt}).  But except few trivial cases, the classes of $M$ for $\p^2$ in the Grothendieck group of varieties are not known.  

Let $M(d,\chi)$ be the moduli space of semistable sheaves of rank 0, first Chern class $dH~(d>0)$ and Euler characteristic $\chi$ on $\p^2$.  $M(d,\chi)\simeq M(d,\chi')$ if $\chi\equiv \pm\chi'$ (mod $d$).
There is a map $\pi:M(d,\chi)\ra |dH|$ sending each sheaf to its support.  Fibers of $\pi$ over integral curves are isomorphic to their (compactified) Jacobians.  But fibers of $\pi$ over non-integral curves are not well understood.

In this paper we build a 1-1 correspondence between pure sheaves of dimension 1 on $\p^2$ and pairs $(E,f)$ with $E$ direct sums of line bundles on $\p^2$ and $f:E\otimes\mmon\hookrightarrow E$ injective, then after putting a stability condition on these pairs we can view $M(d,\chi)$ as the moduli space of semistable pairs $(E,f)$.  From this point of view, we somehow avoid studying fibers of $\pi$ over non-integral curves.  However for a general $d$, the moduli space is still very complicated.  We are only able to describe a big open set of $M(d,\chi)$ with $\chi=1$.  We have the following proposition which is a generalization of Proposition 3.3.1 in \cite{jmdmm} to all multiplicities.
\begin{prop}[\textbf{Proposition \ref{bosih}}]\label{Ino}There is an open subset $W^d\subset M(d,1)$ with $M(d,1)-W^d$ of codimension $\geq2$, and $W^d\simeq \mathbb{P}(\mathcal{V}^d)$, where $\mathcal{V}^d$ is a vector bundle of rank $3d$ over $N_0^d:=Hilb^{[\frac{(d-1)(d-2)}2]}(\p^2)-\Omega_{d-3}^{[\frac{(d-1)(d-2)}2]}$ with $Hilb^{[n]}(\p^2)$ the Hilbert scheme of $n$-points on $\p^2$ and $\Omega_{k}^{[n]}$ the closed subscheme of $Hilb^{[n]}(\p^2)$ parametrizing $n$-points lying on a curve of class $kH.$

\end{prop}  

Denote by $[X]$ the class of a variety $X$ in the Grothendieck ring of varieties.  For $d\leq 5$ and $g.c.d.(d,\chi)=1$, we compute $[M(d,\chi)]$ and get the following three theorems, with $\bl:=[\mathbb{A}^1]$ the class of the affine line.  
\begin{thm}[\textbf{Theorem \ref{mtott}}]\label{Inth}For $d\leq3$, $M(d,1)=W^d$.  Moreover $W^d\simeq|dH|\simeq \p^{3d-1}$ for $d=1,2$; $W^3\simeq \mc_3$ with $\mc_3$ the universal curve in $\p^2\times |3H|$.
\end{thm}
\begin{thm}[\textbf{Theorem \ref{mtf}}]\label{Infou}$[M(4,1)]=\sum_{i=0}^{17}b_{2i}\bl^{i}$ and 
\begin{eqnarray}
&&b_0=b_{34}=1,~~ b_2=b_{32}=2, ~~b_4=b_{30}=6,\nonumber\\&& b_6=b_{28}=10,~b_8=b_{26}=14,~ b_{10}=b_{24}=15,\nonumber\\&& b_{12}=b_{14}=b_{16}=b_{18}=b_{20}=b_{22}=16.\nonumber\end{eqnarray}
In particular the Euler number $e(M(4,1))$ of the moduli space is $192.$ 
\end{thm}
\begin{thm}[\textbf{Theorem \ref{mtfi}}]\label{Infiv}$[M(5,1)]=[M(5,2)]=\sum_{i=0}^{26}b_{2i}\bl^i$ and 
\begin{eqnarray}
&&b_0=b_{52}=1,~~ b_2=b_{50}=2, ~~b_4=b_{48}=6,\nonumber\\&& b_6=b_{46}=13,~b_8=b_{44}=26,~ b_{10}=b_{42}=45,\nonumber\\&& b_{12}=b_{40}=68,~b_{14}=b_{38}=87,~ b_{16}=b_{36}=100,\nonumber\\&& b_{18}=b_{34}=107,~b_{20}=b_{32}=111,~ b_{22}=b_{30}=112,\nonumber\\&& b_{24}=b_{26}=b_{28}=113.\nonumber\end{eqnarray}
In particular the Euler number of both moduli spaces is $1695.$ 
\end{thm}  
\begin{rem}The Euler numbers
$e(M(d,\chi))$ have been computed in \cite{kkv} partially using physics arguments for $M(d,\chi)$ smooth.  They have $e(M(d,\chi))=(-1)^{dim(M(d,\chi))}n^0_d$ with $n^0_d$ so-called BPS states of weight 0 for the local $\p^2$ (see Equation (4.2) and Table 4 in Section 8.3 in \cite{kkv}).  We see that our result accords with theirs for $d\leq 5$.    
\end{rem}
\begin{rem}Theorem \ref{Infou} and Theorem \ref{Infiv} give the motive decompositions of $M(d,r)$ for $d=4,5$, $r$ coprime to $d$.  But according to the result in \cite{yfour} that these moduli spaces admit affine pavings, we also get cell decompositions of them.
\end{rem}
The structure of the paper is arranged as follows.  In Section 2, we construct a 1-1 correspondence between pure sheaves of dimension 1 and pairs $(E,f)$. The stability condition of $(E,f)$ is given in Section 3.  In Section 4 we study the big open set $W^d$ in $M(d,1)$ and prove Proposition \ref{Ino}. 
Theorem \ref{Inth} and Theorem \ref{Infou} are proved in Section 5 while Theorem \ref{Infiv} is proved in the last section---Section 6.  We have Appendix A and B where we prove some technical lemmas used in Section 6.      

\begin{flushleft}{\textbf{Acknowledgments.}}I was partially supported by NSFC grant 11301292.  The first version of the paper was written in 2012 when I was a post-doc at Max-Plank Institute in Bonn.  The paper was revised in 2013 when I was a post-doc at MSC in Tsinghua University in Beijing.  I thank D. Maulik for telling me his prediction from GW/DT/PT theory that the Euler number of $M(d,\chi)$ only depends on $d$ given $g.c.d.(d,\chi)=1$ in 2009 in Hangzhou, which motivated me to study $[M(d,\chi)]$.  Finally I thank L. G\"ottsche, Y. Hu and Z. Zhang for some helpful discussions.
\end{flushleft}      
\section{Pure sheaves of dimension 1 on $\p^2$.}
From now on except otherwise stated, a pair $(E,f)$ on $\p^2$ always satisfies the following two conditions:
\begin{equation}\label{dsl}
(1) E\simeq\bigoplus_i \mo_{\p^2}(n_i)~i.e.E~is~a~direct~sum~of~line~bundles~on~\p^2;
\end{equation}  
\begin{equation}\label{inm}
(2) f\in Hom(E\otimes \mo_{\p^2}(-1),E)~and~moreover~f~is~injective.~~~~~~
\end{equation} 

\begin{defn}\label{doi}We say two pairs $(E,f)$ and $(E',f')$ are isomorphic if $E\simeq E'$ and there exist two isomorphisms $\varphi$ and $\phi$ from $E$ to $E'$ such that the following diagram commutes
\begin{equation}\label{comd}
\xymatrix{
E\otimes \mo_{\p^2}(-1)  \ar[r]^{~~~~~f}\ar[d]_{\varphi\otimes id_{\mmon}}
                & E \ar[d]^{\phi}  \\
                E'\otimes \mo_{\p^2}(-1) \ar[r]^{~~~~~~f'}& E' .
               }
\end{equation}

\end{defn}
 
Define two sets as follows
 \[\mathcal{A}:=\{Isomorphism~classes~of~pure~sheaves~of~dimension~1\};~~~~~~~~~~~~\]
 \[\mathcal{B}:=\{Isomorphism~classes~of~pairs~(E,f)
 \}.~~~~~~~~~~~~~~~~~~~~~~~~~~~~~~~~~~~~\] 
We have a set-map $\z$ from $\mb$ to $\ma$ sending each pair to its cokernel.  We want to prove that $\z$ is bijective.  First we have the following lemma.
\begin{lemma}\label{ioi} Let $F$ be a sheaf of rank 0 and first Chern class $dH~(d>0)$ on $\p^2$,  then $F$ is pure of dimension $1$ if and only if $F$ lies in the following exact sequence with $E_{F}$ a direct sum of line bundles.
\begin{equation}\label{lfr}
0\ra E_{F}\otimes \mo_{\p^2}(-1)\ra E_{F}\ra F\ra 0.
\end{equation}
\end{lemma}
\begin{proof}The ``if" part is obvious: $F$ in (\ref{lfr}) is of rank 0 and has a locally free resolution of length 1,  hence $F$ is pure of dimension 1.  To show the ``only if", it is enough to construct the sequence (\ref{lfr}) for every pure sheaf $F$.  We first follow the construction given by Le Potier in Proposition 3.10 in \cite{lee}.

Denote by $Supp(F)$ the support of $F$.  Since $F$ is a torsion sheaf, we can take a point $x\in\p^2-Supp(F)$.  Let $U:=\p^2-\{x\}$, then $U$ is isomorphic to the total space of $\mo_{\p^1}(1)$ on $\p^1$ with a projection $p:U\ra \p^1$.  $F$ is a sheaf of $\mo_U$-modules.  $F$ is pure, hence $p_{*}F$ is pure and hence of form $\bigoplus_{i}\mo_{\p^1}(n_i)$.  $p_{*}F$ has a structure of $p_{*}\mo_U$-module which gives a morphism $f_1:p_{*}F\ra p_{*}F\otimes\mo_{\p^1}(1)$.  Let $\widetilde{E_F}:=p^{*}(p_{*}F)$.  Pull $f_1$ back to $U$ and define the following morphim
\begin{equation}\label{pb}\tilde{f}:=(p^{*}f_1-\lambda id_{\widetilde{E_F}})\otimes p^{*}id_{\mo_{\p^1}(-1)}:\widetilde{E_F}\otimes p^{*}\mo_{\p^1}(-1)\ra\widetilde{E_F},
\end{equation}
where $\lambda$ is the canonical section of $p^{*}\mo_{\p^1}(1)$.  $\tilde{f}$ is injective and the cokernel is the sheaf $F$. 

On the other hand, $\p^2-U=\{x\}$ is of codimension 2 and $\widetilde{E_F}$ is a direct sum of line bundles on $U$, hence both $\tilde{f}$ and $\widetilde{E_F}$ can be extended to the whole $\p^2$ and we get a resolution of $F$ on $\p^2$ as in (\ref{lfr}) and $E_F\simeq j_{*}\widetilde{E_F}\simeq\bigoplus_{i} \mo_{\p^2}(n_i)$ with $j:U\ra\p^2$ the open immersion.  Hence the lemma.
\end{proof}
\begin{rem}\label{spsh}The form of $E$ in $(E,f)$ is determined by $h^0(F(n))$ with a finite number of $n$, where $F=coker(f)$ and $F(n):=F\otimes\mo_{\p^2}(n)$.  
Moreover $E_1\simeq E_2$ iff $h^0(F_1(n))=h^0(F_2(n))$ for all $n$.   
\end{rem}
Lemma \ref{ioi} implies that $\z$ is surjective, then we have the injectivity by the following lemma.
\begin{lemma}\label{inje} Take any two exact sequences 
\[\xymatrix{0\ar[r] &E_{1}\otimes \mo_{\p^2}(-1)\ar[r]^{~~~~~f_1} &E_{1}\ar[r]^{g_1} &F_1\ar[r] &0 \\
0\ar[r] &E_{2}\otimes \mo_{\p^2}(-1)\ar[r]^{~~~~~f_2} &E_{2}\ar[r]^{g_2} &F_2\ar[r] &0,}\]
with $E_{i}$ direct sums of line bundles,  then $F_1\simeq F_2$ if and only if $(E_{1},f_1)\simeq (E_{2},f_2)$.
\end{lemma} 
\begin{proof}We only need to show the ``only if".  $E_{1}\simeq E_{2}$ if $F_1\simeq F_2$ by Remark \ref{spsh} .  We then want to construct the commutative diagram (\ref{comd}).  $f_i$ can be represented by square matrices with entries in $\bigoplus_{i\geq 0}H^0(\mo_{\p^2}(i))$.  After some invertible transformation, we can ask $f_i$ to have the following form
\[f_i=\left(\begin{array}{cc}\mathbb{I}_{i}&0\\0& \mathtt{T}_i\end{array}\right),\]
with $\mathbb{I}_{i}$ the identity matrix and $\mathtt{T}_i$ a square matrix with entries in $\bigoplus_{i\geq 1}H^0(\mo_{\p^2}(i)).$  Hence we can write $E_i\simeq K_i\oplus M_i$ and $E_i\otimes\mmon\simeq K_i\oplus N_i$, such that $f_i$ splits into the direct sum of an identity on $K_i$ and a morphism $t_i:N_i\ra M_i$ represented by $\mathtt{T}_i$.  We then have the following exact sequence which is a minimal free resolution of $F_i$ (see \cite{bud} Page 5 Definition)
\[\xymatrix{0\ar[r] &N_i\ar[r]^{t_i} &M_i\ar[r]^{g_i|_{M_i}} &F_i\ar[r] &0.}\]

Because of the uniqueness of the minimal free resolution (see \cite{bud} Page 6 Theorem 1.6), we have the following commutative diagram
\begin{equation}\label{cdtw}\xymatrix@C=1.0cm{0\ar[r] &N_1\ar[r]^{t_1}
\ar[d]_{\beta}^{\simeq} 
&M_{1}\ar[r]^{g_1|_{M_1}}\ar[d]_{\alpha}^{\simeq} &F_1\ar[r] \ar[d]^{\simeq}&0 \\
0\ar[r] &N_{2}\ar[r]^{t_2} &M_{2}\ar[r]^{g_2|_{M_2}} &F_2\ar[r] &0.}\end{equation}
Hence we have $K_1\simeq K_2$ because $E_1\simeq E_2$ and $M_1\simeq M_2$.  We define a map $\phi:E_1\ra E_2$ to be $\mathrm{I}_K\oplus \alpha$ with $\mathrm{I}_K$ an isomorphism from $K_1$ to $K_2$, and similarly we define the other map $\varphi\otimes id_{\mmon}:E_1\otimes\mmon\ra E_2\otimes\mmon$ to be $\mathrm{I}_K\oplus \beta$.  Then we have the following commutative diagram
\begin{equation}\label{cdth}\xymatrix{0\ar[r] &E_{1}\otimes \mo_{\p^2}(-1)\ar[r]^{~~~~~f_1}
\ar[d]_{\varphi\otimes id_{\mmon}}^{\simeq} 
&E_{1}\ar[r]^{g_1}\ar[d]_{\phi}^{\simeq} &F_1\ar[r] \ar[d]^{\simeq}&0 \\
0\ar[r] &E_{2}\otimes \mo_{\p^2}(-1)\ar[r]^{~~~~~f_2} &E_{2}\ar[r]^{g_2} &F_2\ar[r] &0.}\end{equation}
This finishes the proof of the lemma.   
\end{proof}
We finally get the following proposition.
\begin{prop}\label{oto}There is a 1-1 correspondence between isomorphism classes of pure sheaves of dimension 1 and isomorphism classes of pairs $(E,f)$. 
\end{prop}

\section{The stability condition.}

We put a stability condition on our pairs $(E,f)$, so that the map $\z$ induces a bijection from semistable pairs to semistable sheaves.
Given a pair $(E,f)$ and its image $F$ via $\z$, we write down the exact sequence 
\begin{equation}\label{lfrt}\xymatrix{0\ar[r] &E\otimes\mmon\ar[r]^{~~~~~~~f}&E\ar[r]^{g}&F\ar[r]&0}.
\end{equation}

Recall that the slop of a torsion free sheaf $E$, $\mu(E)$, is defined as follows
\[\mu(E):=\frac{deg(E)}{rank(E)};\] 
and for a sheaf $F$ of dimension 1 we have 
\[\mu(F):=\frac{\chi(F)}{deg(F)}.\]
We then have $\mu(E)+1=\mu(F)$ for $E,F$ in the sequence (\ref{lfrt}).
\begin{defn}We say a pair $(E,f)$ is (semi)stable if for any subsheaf $E'\subsetneq E$ and $E'$ a direct sum of line bundles such that $f^{-1}(E')\simeq E'\otimes\mmon$,  we have $\mu(E')(\leq)<\mu(E)$.
\end{defn}
\begin{lemma}\label{ssi}$\z$ induces a bijection from semistable pairs to semistable sheaves.
\end{lemma}
\begin{proof}Look at the sequence (\ref{lfrt}).  To prove the lemma, we only need to prove that $\forall F'\subsetneq F$,  $\exists E'$ a direct sum of line bundles and $f^{-1}(E')\simeq E'\otimes\mmon$, such that $F\simeq coker(f|_{f^{-1}(E')})$.  Keep notations the same as in the proof of Lemma \ref{ioi}, and we see that $p_{*}F'$ is a direct sum of line bundles on $\p^1$ and $p_{*}F'\subset p_{*}F$.  By following the construction in the proof of Lemma \ref{ioi}, we get $E'\simeq j_{*}p^{*}p_{*}F'$ a direct sum of line bundles and moreover $f^{-1}(E')\simeq E'\otimes\mmon$, hence the lemma.\end{proof}

\begin{rem}\label{mni}By Remark \ref{spsh}, we see that $E_1$ might not isomorphic to $E_2$ while $F_1$ is only S-equivalent to $F_2$.\end{rem}

From now on let $H$ be the hyperplane class in $\p^2$ and $u_{d,\chi}$ the class in the Grothendieck group of coherent sheaves on $\p^2$, which is of rank 0, first Chern class $dH$ and Euler characteristic $\chi$.  Denote by $M(d,\chi)$ the moduli space parametrizing semistable sheaves of class $u_{d,\chi}$.  Then $M(d,\chi)$ is irreducible (see \cite{lee} Theorem 3.1) and the stable locus $M(d,\chi)^s$ is smooth.  $M(d,\chi)\simeq M(d,\chi')$ if $\chi\equiv\pm\chi'$ (mod $d$).  $M(d,\chi)=M(d,\chi)^s$ if and only if $g.c.d.(d,\chi)=1$.  Hence $M(d,1)=M(d,1)^s$ and the moduli space is smooth of dimension $d^2+1$ for all $d\geq 1$.  Moreover there is a universal sheaf on $M(d,1)\times\p^2$ by Theorem 3.19 in \cite{lee}.

Let $g.c.d(d,\chi)=1$, then we can assign every point $F$ in $M(d,\chi)$ uniquely to a pair $(E,f)$ such that $rank(E)=d$ and $c_1(E)=(\chi-d)H$.   We view every point in $M(d,\chi)$ as a pair $(E,f)$ and stratify $M(d,\chi)$ by the form of $E$, then every stratum is a constructible set in $M(d,\chi).$ 

We write down the following lemma for future use.
\begin{lemma}\label{ncss}Let $(E,f)$ be a semistable pair.  Let $D',D''$ be two direct summands of $E$ such that $D'\simeq D''$ and $f(D'\otimes\mmon)\subset D''$.  Then we have $\mu(D')\leq \mu(E)$.  In particular $E$ must have the form $\oplus_{i=1}^k\mo_{\p^2}(n+i)^{\oplus a_i}$ with $n$ some integer and $a_i>0$ for all $1\leq i\leq k$. 
\end{lemma}
\begin{proof}Since $D'\simeq D''$ and they are both direct summands of $E$, $f(D'\otimes\mmon)\subset D''\Leftrightarrow f^{-1}(D'')=D'\otimes\mmon\simeq D''\otimes\mmon$.  Hence the first statement.

Write $E\simeq \oplus_{i=1}^k\mo_{\p^2}(n_i)^{\oplus a_i}$ with $a_i>0$ and $n_1>n_2>\ldots >n_k$, then we want to show that $n_i-n_{i+1}=1$ for all $1\leq i\leq k-1$.  Assume $\exists i_0$ such that $n_{i_0}-n_{i_0+1}\geq 2$, then $f((\oplus_{i=1}^{i_0}\mo_{\p^2}(n_i)^{\oplus a_i})\otimes\mmon)\subset \oplus_{i=1}^{i_0}\mo_{\p^2}(n_i)^{\oplus a_i}$ and $\mu(\oplus_{i=1}^{i_0}\mo_{\p^2}(n_i)^{\oplus a_i})>\mu(E)$, which is a contradiction.  Hence the statement.   
\end{proof}

\section{A big open subset in $M(d,1)$.}
We want to give a concrete description of an open subset in $M(d,1)$, where the pairs $(E,f)$ satisfy that $E\simeq\mo_{\p^2}\oplus\mmon^{\oplus d-1}$.  We first have the following lemma.
\begin{lemma}\label{roc}If $E\simeq\mo_{\p^2}\oplus\mmon^{\oplus d-1}$, then $(E,f)$ is stable if and only if for any two direct summands $D',D''$ of $E$ such that $D'\simeq D''$ and $f(D'\otimes\mmon)\subset D''$,  we have $\mu(D')<\mu(E)$.
\end{lemma}
\begin{proof}The lemma is equivalent to Claim 4.2 in \cite{mmz}.  We also prove it here.
Because of Lemma \ref{ncss}, we only need to prove the ``if".  By direct observation we see that if $E'\subset E$ is a direct sum of line bundles such that $\mu(E')> \mu(E)$, then $E'$ is a direct summand of $E$.  Hence the lemma.  
\end{proof}

Define $\widetilde{W}^d:=\{(E,f)|E\simeq\mo_{\p^2}\oplus\mmon^{\oplus d-1}\}\subset M(d,1)$, we then have the following lemma.
\begin{lemma}\label{codt}$M(d,1)-\widetilde{W}^d$ is of codimension $\geq 2$ in $M(d,1)$.
\end{lemma}
\begin{proof}For any point $x\in\p^2$, denote $Y_x$ to be the open subset of $M(d,1)$ where the pair $(E,f)$ satisfies that $x\not\in Supp(coker(f))$.  $M(d,1)$ can be covered by finitely many $Y_x$.  According to Proposition 3.14 in \cite{lee}, $Y_x\cap(M(d,1)-\widetilde{W}^d)$ is of codimension $\geq 2$ in $Y_x$, hence the lemma. 
\end{proof}
Now we look at a pair $(E,f)$ in $\widetilde{W}^d$.  $f(\mo_{\p^2}\otimes\mmon)\not\subset \mo_{\p^2}$ by Lemma \ref{roc}, hence the restriction $f_{restr}:\mo_{\p^2}\otimes\mmon\ra\mmon^{\oplus d-1}$ is nonzero.  Therefore we can ask $f$ to identify $\mo_{\p^2}\otimes\mmon$ with a summand $\mmon$ and then $f$ can be represented by the following matrix
\[\left(\begin{array}{ccc}0&1&0\\ A&0 &B\end{array}\right),\]
where $A$ is a $(d-1)\times 1$ matrix with entries in $H^0(\mo_{\p^2}(2))$ and $B$ a $(d-1)\times (d-2)$ matrix with entries in $H^0(\mone)$.  $B$ provides a morphism $f_{B}:\mmon^{\oplus d-1}\ra\mo_{\p^2}^{\oplus d-2}$.  By Lemma \ref{roc} the stability condition is equivalent to the following condition
\begin{cond}\label{scom}For any $1<d'<d$, $f_{B}(\mmon^{\oplus d'-1})\nsubseteq\mo_{\p^2}^{\oplus d'-2}$. \end{cond} 

Let $f_{B^t}:\mmon^{\oplus d-2}\ra\mo_{\p^2}^{\oplus d-1}$ be a morphism represented by the transform of $B,B^t$.  Then Condition \ref{scom} is equivalent to the following condition
\begin{cond}\label{scodm}For any $1<d''\leq d-2$, $f_{B^t}^{-1}(\mo_{\p^2}^{\oplus d''})\not\simeq\mmon^{\oplus d''}.$\end{cond}

We have the following diagram
\begin{equation}\label{boscd}\xymatrix{ 0\ar[r]&\mmon^{\oplus d-2}\ar[r]^{~~~f_{B^t}}&\mo_{\p^2}^{\oplus d-1}\ar[r]^{f_q}& Q_f\ar[r]&0\\
&&\mo_{\p^2}(-2)\ar[u]^{f_{A^t}}\ar[ru]_{\sigma_f:=f_q\circ f_{A^t}}&&}.\end{equation}
The injectivity of $f_{B^t}$ is because of the injectivity of $f$.  Let $F:=coker(f)$ and $F^{\vee}:=\mathcal{E}xt^1(F,\mo_{\p^2})$, then $F^{\vee\vee}\simeq F$ and moreover $F$ and $F^{\vee}$ are determined by each other (see \cite{yuan} Lemma A.0.13).
We write down a commutative diagram as follows
\begin{equation}\label{bigcd}\xymatrix@C=0.3cm@R=0.6cm{& &0\ar[d] &0\ar[d]\\
0\ar[r]&\mmon^{\oplus d-2}\ar[r]\ar[d]^{\simeq}&\mmon^{\oplus d-2}\oplus \mo_{\p^2}(-2)\ar[d]^{f_{B^t}\oplus f_{A^t}}\ar[r]&\mo_{\p^2}(-2)\ar[r]\ar[d]^{\sigma_f}&0\\
0\ar[r]&\mmon^{\oplus d-2}\ar[r]^{f_{B^t}}&\mo_{\p^2}^{\oplus d-1}\ar[d]\ar[r]^{f_q}& Q_f\ar[r]\ar[d]^{\delta}&0\\
& &F^{\vee}\otimes \mo_{\p^2}(-2)\ar[r]^{\simeq}\ar[d] &F^{\vee}\otimes\mo_{p^2}(-2)\ar[d]\\& &0 &0}
\end{equation}
We see that the isomorphism class of $F^{\vee}$ is determined by the pair $(Q_f,\sigma_f),$ hence so is the isomorphism class of $F$.  

Define $W^d:=\{[(E,f)]\in \widetilde{W}^d|Q_f~is~torsion~free\}.$  Then we have the following proposition.
\begin{prop}\label{bosih}$W^d\simeq \mathbb{P}(\mathcal{V}^d)$, where $\mathcal{V}^d$ is a vector bundle of rank $3d$ over $N_0^d:=Hilb^{[\frac{(d-1)(d-2)}2]}(\p^2)-\Omega_{d-3}^{[\frac{(d-1)(d-2)}2]}$ with $Hilb^{[n]}(\p^2)$ the Hilbert scheme of $n$-points on $\p^2$ and $\Omega_{k}^{[n]}$ the closed subscheme of $Hilb^{[n]}(\p^2)$ parametrizing $n$-points lying on a curve of class $kH.$

\end{prop}
\begin{proof}Condition \ref{scodm} is satisfied automatically for $Q_f$ torsion free.  Hence $W^d$ consists of all the pairs $(Q_f,\sigma_f)$ in diagram $(\ref{boscd})$ with $Q_f$ torsion free.  

Define $\bar{d}:=\frac{(d-1)(d-2)}2$.  If $Q_f$ in diagram (\ref{boscd}) is torsion free, then by direct calculation we know that $Q_f\simeq I_{\bar{d}}\otimes\mo_{\p^2}(d-2)$, with $I_{n}$ the ideal sheaf of a 0-dimensional subscheme of length $n$ on $\p^2$.  

The following lemma shows that $N_0^d$ parametrizes all the torsion free $Q_f$ in diagram (\ref{boscd}).
\begin{lemma}\label{tfid}$I_{\bar{d}}(d-2):=I_{\bar{d}}\otimes\mo_{\p^2}(d-2)$ has the following resolution 
\begin{equation}\label{rfi}\xymatrix@C=0.5cm{ 0\ar[r]&\mmon^{\oplus d-2}\ar[r]&\mo_{\p^2}^{\oplus d-1}\ar[r]& I_{\bar{d}}(d-2)\ar[r]&0}.\end{equation}
 if and only if $H^0(I_{\bar{d}}(d-3))=0$.  \end{lemma}
\begin{proof}The lemma is equivalent to Proposition 4.5 in \cite{jmd} and it is also a straightforward consequence of Corollary 3.9 in \cite{bud} Page 38 and Proposition 3.1 in \cite{bud} Page 32.
\end{proof}
Up to scalars, $\sigma_f$ can be viewed as an element in $\mathbb{P}H^0(Q_f(2))$ which is exactly $det(f)$. 

Because $N_0^d$ is open in $Hilb^{[\bar{d}]}(\p^2)$, on $\p^2\times N_0^d$ we have a universal sheaf $\mathcal{I}_{\bar{d}}$ which restricted to the fiber over each point $[I_{\bar{d}}]\in N^d_0$ is the ideal sheaf $I_{\bar{d}}$.  We have the diagram
\begin{equation}\label{ufi}\xymatrix{\mathcal{I}_{\bar{d}}\ar[r] &\p^2\times N_0^d\ar[ld]^q\ar[d]^p\\
\p^2 &N_0^d.}\end{equation}
$R^ip_{*}(\mathcal{I}_{\bar{d}}\otimes q^{*}\mo_{\p^2}(d))=0$ for $i\geq1$ and $p_{*}(\mathcal{I}_{\bar{d}}\otimes q^{*}\mo_{\p^2}(d))$ is locally free of rank $3d$.  Define $\mathcal{V}^d:=p_{*}(\mathcal{I}_{\bar{d}}\otimes q^{*}\mo_{\p^2}(d))$.  There is a 1-1 correspondence between points in $\mathbb{P}(\mathcal{V}^d)$ and isomorphism classes of $(Q_f,\sigma_f)$ with $Q_f$ torsion free.  To prove the proposition, it is enough to construct a family $\mathcal{F}$ of stable sheaves of class $u_{d,1}$ over $\p^2\times\mathbb{P}(\mathcal{V}^d)$.

We have the following commutative diagram
\begin{equation}\label{cadph}\xymatrix@C=1.2cm{\p^2\times\p(\mathcal{V}^d)\ar[r]^{id_{\p^2}\times\pi}\ar[d]_{\tilde{p}} &\p^2\times N_0^d\ar[rd]^q\ar[d]_p\\
\p(\mathcal{V}^d)\ar[r]^{\pi} &N_0^d&\p^2}
\end{equation}
Denote by $\mo_{\pi}(1)$ the relative polarization on $\p(\mathcal{V}^d)$ over $N_0^d.$   We have the following exact sequence on $\p^2\times \p(\mathcal{V}^d)$
\begin{equation}\label{ufph}0\ra\mo_{\p^2\times \p(\mathcal{V}^d)}\ra\tilde{p}^{*}\mo_{\pi}(1)\otimes (id_{\p^2}\times\pi)^{*}(\mathcal{I}_{\bar{d}}\otimes q^{*}\mo_{\p^2}(d))\ra\mathcal{F}^{\vee}\ra0.
\end{equation}
We see that fiberwise (\ref{ufph}) is the first vertical exact sequence from the right hand side in (\ref{bigcd}) tensored by $\mo_{\p^2}(2)$.  Hence $\mathcal{F}^{\vee}$ is a family of stable sheaves of class $u_{d,1}^{\vee}$.  We get $\mathcal{F}$ by taking the dual.  Hence the proposition. 
\end{proof}
We now have a concrete description of $W^d$.  
\begin{prop}\label{codth}$M(d,1)-W^d$ is of codimension $\geq 2$ in $M(d,1)$.
\end{prop}
\begin{proof}By Lemma \ref{codt}, we only need to show that $\widetilde{W}^d-W^d$ is of codimension at least 2 in $M(d,1)$.  Denote by $|dH|$ the linear system of divisors of class $dH$,  then non-integral curves form a closed subset of codimension $\geq2$ in $|dH|$.  Therefore by Proposition 2.8 and Lemma 3.2 in \cite{lee}, we know that stable sheaves with non-integral supports form a closed subset of codimension $\geq2$ in $M(d,1)$.  We then want to show that if $Q_f$ in (\ref{boscd}) is not torsion free, then $Supp(F)=Supp(coker(f))$ is non-integral.

Denote by $T_f$ the torsion of $Q_f$. Since $Q_f$ has a free resolution of length 1, $T_f$ must be a pure sheaf supported on a curve in $|d'H|$.  Look back to the diagram (\ref{bigcd}), the map $\delta$ restricted to $T_f$ gives a nonzero element in Hom$(T_f,F^{\vee}\otimes\mo_{\p^2}(-2))$.  If $d'<d$, then $Supp(F^{\vee})=Supp(F)$ can not be integral.  Now assume $d'=d$ and we look at the following exact sequence
\[0\ra T_f\ra Q_f\ra Q_{f}^{tf}\ra 0.\]
The torsion free sheaf $Q^{tf}_f$ has the form $I_n(m):=I_n\otimes \mo_{\p^2}(m)$ such that $m+d'=deg(Q_f)=d-2$.  On the other hand, the surjective morphism $f_q$ induces a surjective morphism from $\mo_{\p^2}^{\oplus d-1}$ to $Q_f^{tf}$, hence $m\geq 1$ and thus $d'<d-2<d$ which is a contradiction.  This finishes the proof.
\end{proof}

\section{$M(d,1)$ with $d\leq 4.$}

In this section we study $M(d,1)$ with $d\leq 4$.  Notice that for $d\leq 4$, up to isomorphism $M(d,1)$ is the only moduli space such that there is no strictly semistable locus, since $M(d,\chi)\simeq M(d,\chi')$ if $\chi\equiv\pm\chi'$ (mod $d$).

For $d\leq3$, $M(d,1)$ is very easy to understand and the following theorem is already known by Theorem 3.5 and Theorem 5.1 in \cite{lee}.  But however using our new description we give another proof.  Recall that we have defined a big open subset $W^d\subset M(d,1)$ in the previous section.
\begin{thm}\label{mtott} For $d\leq3$, $M(d,1)=W^d$.  Moreover $W^d\simeq |dH|\simeq\p^{3d-1}$ for $d=1,2$; $W^3\simeq \mc_3$ with $\mc_3$ the universal curve in $\p^2\times |3H|$.
\end{thm}
\begin{proof}By Lemma \ref{ncss} we see that for $d\leq 3$, the sheaf $E$ in a stable pair $(E,f)$ can only have the form $\mo_{\p^2}\oplus\mmon^{\oplus d-1}$.  From the proof of Proposition \ref{codth} we see that the torsion of $Q_f$ can only be supported on a curve of degree no bigger than $d-3$, hence $Q_f$ is always torsion free for $d\leq3$. Hence the first statement.  By direct observation, we get the form of $W^d$ for $d\leq 3$.  This finishes the proof. 
\end{proof}
Denote by $[X]$ the class of a variety $X$ in the Grothendieck ring of varieties.  Define $\bl:=[\mathbb{A}^1]$ with $\mathbb{A}^1$ the affine line.  We have the following theorem. 
\begin{thm}\label{mtf}$[M(4,1)]=\sum_{i=0}^{17}b_{2i}\bl^{i}$ and  
\begin{eqnarray}
&&b_0=b_{34}=1;~~ b_2=b_{32}=2; ~~b_4=b_{30}=6;\nonumber\\&& b_6=b_{28}=10;~b_8=b_{26}=14;~ b_{10}=b_{24}=15;\nonumber\\&& b_{12}=b_{14}=b_{16}=b_{18}=b_{20}=b_{22}=16.\nonumber\end{eqnarray}
In particular the Euler number $e(M(4,1))$ of the moduli space is $192.$ 
\end{thm}
To prove Theorem \ref{mtf}, we first define two strata as follows.
\begin{eqnarray} & M_1:=&\{[(E,f)]\in M(4,1)|E\simeq\mo_{\p^2}\oplus\mmon^{\oplus 3}\};\nonumber\\& M_2:=&\{[(E,f)]\in M(4,1)|E\simeq\mo_{\p^2}^{\oplus 2}\oplus\mmon\oplus\mo_{\p^2}(-2)\}.\nonumber
\end{eqnarray}
According to Lemma \ref{ncss}, $M(4,1)=M_1\sqcup M_2.$
\begin{lemma}\label{scff}A pair $(E,f)$ with $rank(E)=4$ and $deg(E)=-3$ is stable if and only if for any two direct summands $D',D''$ of $E$ such that $D'\simeq D''$ and $f(D'\otimes\mmon)\subset D''$,  we have $\mu(D')<\mu(E)$.
\end{lemma}
\begin{proof}We only need to prove the lemma for $E\simeq\mo_{\p^2}^{\oplus 2}\oplus\mmon\oplus\mo_{\p^2}(-2)$.  We want to show that if $\exists E'\subset E$, $E'$ is a direct sum of line bundles with $\mu(E')>\mu(E)$ and $f^{-1}(E')\simeq E'\otimes\mmon$, then $\exists D,D'\subset E$ two direct summands with $D\simeq D'$ and $\mu(D)>\mu(E)$, such that $f(D\otimes\mmon)\subset D'$.  With no loss of generality, we assume that $E'$ has the form $\bigoplus_{i}\mo_{\p^2}(n_i)^{\oplus a_i}$ with $a_i>0$ and $n_{i}-n_{i+1}=1$.  

Let $E'\simeq E''\subset E$ and $E''$ is not a direct summand of $E$.  Then $E''$ has to be one of the following two cases:

(1) $E''\subset\mo_{\p^2}^{\oplus 2}$ and $E''\simeq \mo_{\p^2}\oplus\mmon$;

(2) $E''\subset\mo_{\p^2}^{\oplus 2}\oplus\mmon$ and $E''\simeq \mo_{\p^2}\oplus\mmon^{\oplus 2}.$

Let $E''$ be in case (1).  If $E'$ is a direct summand of $E$, then $f^{-1}(E')\neq E''\otimes\mmon$ because by Nakayama's lemma $E''\otimes\mmon\subset f^{-1}(E')\Rightarrow \mo_{\p^2}^{\oplus 2}\otimes\mmon\subset f^{-1}(E')$ which contradicts the injectivity.  

If $E'=E''$ and $f^{-1}(E')=E''\otimes\mmon$, then again by Nakayama's lemma we have $f(\mo_{\p^2}^{\oplus 2}\otimes\mmon)\subset\mo_{\p^2}^{\oplus 2}.$  Hence we get $D=D'=\mo_{\p^2}^{\oplus 2}$.

If $E'=E''$ and $f^{-1}(E')$ is a direct summand of $E\otimes\mmon$ isomorphic to $E'\otimes \mmon$, then we have $D=\mo_{\p^2}^{2}\oplus \mmon=D'$ and $f(D\otimes\mmon)\subset D'$, since there is no nonzero morphism from $\mo_{\p^2}\otimes\mmon$ to $\mo_{\p^2}(-2)$.  Hence case (1) is done.

Case (2) is analogous and this finishes the proof of the lemma.
\end{proof}
For a pair $(E,f)\in M_2$, $f$ can be represented by the following matrix
\begin{equation}\label{romf}\left(\begin{array}{cccc}b_1&b_2&0&0\\0&0&1&0\\0&0&0&1\\a_1&a_2&0&0\end{array}\right),
\end{equation}
where $b_i\in H^0(\mone)$ and $a_i\in H^0(\mo_{\p^2}(3))$.  The injectivity of $f$ implies that $det(f)=b_1a_2-b_2a_1\neq 0$.  Moreover by Lemma \ref{scff} $(E,f)$ is stable if and only if $kb_1\neq k'b_2$ for any $(k,k')\in\bc^2-\{0\}$.  
\begin{lemma}\label{mcfmt}$[M_2]=[\p^2\times\p^{13}].$
\end{lemma}
\begin{proof}We have the following diagram
\begin{equation}\label{ffcd}\xymatrix{ 0\ar[r]&\mmon\ar[r]^{(b_1,b_2)}&\mo_{\p^2}^{\oplus 2}\ar[r]^{f_r}& R_f\ar[r]&0\\
&&\mo_{\p^2}(-3)\ar[u]^{(a_1,a_2)}\ar[ru]_{\omega_f:=f_r\circ(a_1,a_2)}&&}.\end{equation}
Since $kb_1\neq k'b_2$ for any $(k,k')\in\bc^2-\{0\}$, we have $R_f\simeq I_1(1)$ and every $I_1(1)$ can be put in sequence (\ref{ffcd}).  Hence $M_2$ consists of all the isomorphism classes of pairs $(R_f,\omega_f)$ with $R_f\simeq I_1(1)$.  Hence $M_2$ is isomorphic to a projective bundle over $Hilb^{[1]}(\p^2)$ with fibers isomorphic to $\p(H^0(I_1(4)))\simeq \p^{13}$. Hence the lemma.
\end{proof}

The big open subset $W^4$ defined in the previous section is contained in $M_1$.  We have $[W^4]=[(Hilb^{[3]}(\p^2)-\Omega^{[3]}_1)\times \p^{11}]$ by Proposition \ref{bosih}.
\begin{lemma}\label{doof}$\Omega^{[3]}_1\simeq \mc_{1}^{[3]}$ with $\mc_{1}^{[3]}$ the relative Hilbert scheme of $3$-points on the universal family $\mc_1\subset\p^2\times |H|$, and hence $[\Omega^{[3]}_1]=[\p^2\times\p^3]$.
\end{lemma}
\begin{proof}We have a natural map $\xi:\mc_1^{[3]}\ra\Omega^{[3]}_1$. $\xi$ is an isomorphism because there is at most one curve in $|H|$ passing through any $3$ points.  $\mc_1\ra |H|$ is a $\p^1$-bundle, hence the map $p:\mc_1^{[3]}\ra|H|$ is a projective bundle with fibers isomorphic to $(\p^1)^{[3]}\simeq\p^3$, therefore $[\Omega^{[3]}_1]=[\mc_1^{[3]}]=[\p^2\times\p^3].$  
\end{proof}
Now we want to compute $[M_1-W^4]$.  Look back to diagram (\ref{boscd}), we want to see what $Q_f$ will be if it is not torsion free for $d=4$.  We know that the torsion of $Q_f$ can only be supported on a curve of degree no bigger than $d-3=1$ (See the proof of Proposition \ref{codth}).  We write down the following exact sequence 
\[0\ra T_f\ra Q_f\ra Q^{tf}_f\ra0,\]
with $T_f$ the torsion of $Q_f$ and $Q^{tf}_f$ a torsion free sheaf of rank 1.  

Since $T_f$ is supported on a curve in $|H|$ and $h^0(T_f\otimes\mmon)\leq h^0(Q_f\otimes\mmon)=0$, $T_f\simeq\mo_{H}(t)\simeq\mo_{\p^1}(t)$ with $t\leq 0$.  Let $Q^{tf}_f\simeq I_{n}(m)$ with $m>0,n\geq0$.  Then we have $m=1$ and $n-t=1$ by direct calculation.  

If $t=0,n=1$, then we have the following commutative diagram
\begin{equation}\label{ssff}\xymatrix@R=0.3cm{ &0\ar[d]&0\ar[d] &0\ar[d] &\\0\ar[r]&\mmon\ar[r]\ar[d]&\mo_{\p^2}\ar[r]\ar[d]& \mo_{H}\ar[r]\ar[d]&0\\0\ar[r]&\mmon^{\oplus 2}\ar[r]\ar[d]&\mo_{\p^2}^{\oplus 3}\ar[r]\ar[d]& Q_f\ar[r]\ar[d]&0\\
0\ar[r]&\mmon\ar[r]\ar[d]&\mo_{\p^2}^{\oplus 2}\ar[r]\ar[d]& I_1(1)\ar[r]\ar[d]&0\\ &0&0&0&}\end{equation}
which contradicts Condition \ref{scodm}.  Hence we have $t=-1,n=0$ and $Q_f$ lies in the following exact sequence
\begin{equation}\label{tqf}0\ra\mo_{H}(-1)\ra Q_f\ra \mo_{\p^2}(1)\ra 0.\end{equation}
Ext$^1(\mo_{\p^2}(1),\mo_{H}(-1))\simeq H^1(\mo_{\p^1}(-2))\simeq\bc$, so for a fixed projective line $\p^1$ of class $H$, if (\ref{tqf}) does not split, $Q_f$ is unique up to isomorphism.
\begin{lemma}\label{tqns}$Q_f$ in (\ref{tqf}) also lies in the following exact sequence (\ref{lfrftq}) if and only if (\ref{tqf}) does not split.
\begin{equation}\label{lfrftq}0\ra\mmon^{\oplus 2}\ra\mo_{\p^2}^{\oplus 3}\ra Q_f\ra 0.
\end{equation} 
\end{lemma}
\begin{proof}If $Q_f\simeq\mo_{H}(-1)\oplus\mone$, it certainly can not lie in (\ref{lfrftq}).  If the sequence (\ref{tqf}) does not split, then $Q_f$ is unique, so we only need to construct the sequence (\ref{lfrftq}) with $Q_f$ contains $\mo_{H}(-1)$ as its torsion.  Write $\p^2=\mathtt{Proj}\bc[x_0,x_1,x_2]$.  With no loss of generality we assume that $\mo_{H}(-1)$ is supported on $\{x_0=0\}$, then the following matrix represents a morphism $f_{B^t}:\mmon^{\oplus 2}\ra\mo_{\p^2}^{\oplus 3}$ such that $coker(f_{B^t})$ contains $\mo_{\{x_0=0\}}(-1)$ as its torsion.
\begin{equation}\label{cmtq}f_{B^t}:=\left(\begin{array}{cc}x_0&0\\x_1&x_2\\0&x_0\end{array}\right).\end{equation}
This finishes the proof.
\end{proof}
\begin{rem}\label{scsf}$f_{B^t}$ defined in (\ref{cmtq}) also satisfies the stability condition i.e. Condition \ref{scodm}.  \end{rem}
\begin{lemma}\label{icqff}Decompose $|H|$ into cells and write $|H|=\displaystyle{\cup_{i=0}^2}\mathbb{A}^i$.  Then $\mathbb{A}^i$ parametrizes isomorphism classes of $Q_f$ such that there are pairs $(Q_f,\sigma_f)\in M_1-W^4$ and $T_f$ are supported on curves in $\mathbb{A}^i\subset |H|$.
\end{lemma}
\begin{proof}Lemma \ref{tqns} implies that there is a 1-1 correspondence between isomorphism classes of $Q_f$ and points in $|H|$.  We need to decompose $|H|$ into cells so that we have a universal family over $\p^2\times \mathbb{A}^i$ for each $i$.  

We have the following diagram
\begin{equation}\label{ucdo}\xymatrix{\mc_1\ar[r]&\p^2\times |H|\ar[ld]_q\ar[d]^p\\ \p^2 & |H|},\end{equation}
with $\mc_1$ the universal curve of degree 1.

Ext$^i(\mone,\mo_{H}(-1))=0$ for all $i\neq 1$ and Ext$^1(\mone,\mo_{H}(-1))\simeq\bc$, therefore $L:=\mathcal{E}xt^1_p(q^{*}\mone,\mo_{\mc_1}\otimes q^{*}\mmon)$ is a line bundle on $|H|$.  Moreover $ch(L)=-ch(R^{\bullet}p_{*}\circ R^{\bullet}\mathcal{H}om (q^{*}\mone,\mo_{\mc_1}\otimes q^{*}\mmon))$, so by Grothendieck-Hirzbruch-Riemann-Roch Theorem we can compute and get that $c_1(L)=c_1(\mo_{|H|}(1))$ and hence $L\simeq \mo_{|H|}(1)$.

$L$ has a nowhere vanishing global section on each $\mathbb{A}^i$, in other words, we have an exact sequence on $\p^2\times \mathbb{A}^i$ 
\begin{equation}\label{uftqf}0\ra \mo_{\mc_1}\otimes q^{*}\mmon|_{\p^2\times\mathbb{A}^i}\ra\mathcal{Q}^i\ra q^{*}\mone|_{\p^2\times\mathbb{A}^i}\ra 0,
\end{equation}
such that restricted on the fiber over any point $y\in \mathbb{A}^i$ it does not split.  Hence $\mathcal{Q}^i$ is the family we want and hence the lemma.
\end{proof}  
We rewrite diagram (\ref{boscd}) for $d=4$ as the following diagram
\begin{equation}\label{tqfcd}\xymatrix{ 0\ar[r]&\mmon^{\oplus 2}\ar[r]^{~~f_{B^t}}&\mo_{\p^2}^{\oplus 3}\ar[r]^{f_q}& Q_f\ar[r]&0\\
&&\mo_{\p^2}(-2)\ar[u]^{f_{A^t}}\ar[ru]_{\sigma_f=f_q\circ f_{A^t}}&&}.\end{equation}

Denote $\p(p_{*}(\mathcal{Q}\otimes q^{*}\mo_{\p^2}(2)))$ to be the union of the projective bundles $\p(p_{*}(\mathcal{Q}^i\otimes q^{*}\mo_{\p^2}(2)|_{\p^2\times\mathbb{A}^i}))$ over $\mathbb{A}^i$ with fibers isomorphic to $\p H^0(Q_f(2))\simeq \p^{11}$.  Analogously, isomorphism classes of pairs $(Q_f,\sigma_f)$ can be parametrized by $\p(p_{*}(\mathcal{Q}\otimes q^{*}\mo_{\p^2}(2)))$.  However $\p(p_{*}(\mathcal{Q}\otimes q^{*}\mo_{\p^2}(2)))\not\subset M(4,1)$.  Look back to diagram (\ref{bigcd}), it is easy to see that 
\begin{equation}\label{insff}\p(p_{*}(\mathcal{Q}\otimes q^{*}\mo_{\p^2}(2)))\cap M(4,1)=\{[(Q_f,\sigma_f)]|Im(\sigma_f)\not\subset T_f\},\end{equation}
with $Im(\sigma_f)$ the image of $\sigma_f$ and $T_f$ the torsion of $Q_f$. 

The complement of $(\ref{insff})$ in $\p(p_{*}(\mathcal{Q}\otimes q^{*}\mo_{\p^2}(2)))$ is the union of the projective bundles $\p(p_{*}(\mo_{\mc_1}\otimes q^{*}\mone|_{\p^2\times\mathbb{A}^i}))$ over $\mathbb{A}^i$ with fibers isomorphic to $\p H^0(\mo_{H}(1))\simeq\p^1.$  The following lemma is a straightforward consequence.
\begin{lemma}\label{lpff}$[M_1-W^4]=[\p^2\times\p^{11}-\p^2\times\p^1].$
\end{lemma}  
\begin{proof}[Proof of Theorem \ref{mtf}]By Lemma \ref{mcfmt}, Lemma \ref{doof} and Lemma \ref{lpff}, we have
\[[M(4,1)]=[(Hilb^{[3]}(\p^2)-\p^2\times\p^3)\times\p^{11}+\p^2\times(\p^{11}-\p^1)+\p^2\times\p^{13}],\]
which leads to the theorem by direct calculation. 
\end{proof}

\section{$M(5,1)$ and $M(5,2)$.}
Up to isomorphism $M(5,1)$ and $M(5,2)$ are the only two moduli spaces with $d=5$ such that there is no strictly semistable locus.  In this section we prove the following  theorem.
\begin{thm}\label{mtfi}$[M(5,1)]=[M(5,2)]=\sum_{i=0}^{26}b_{2i}\bl^i$ and 
\begin{eqnarray}
&&b_0=b_{52}=1,~~ b_2=b_{50}=2,~~b_4=b_{48}=6,\nonumber\\&& b_6=b_{46}=13,~b_8=b_{44}=26,~ b_{10}=b_{42}=45,\nonumber\\&& b_{12}=b_{40}=68,~b_{14}=b_{38}=87,~ b_{16}=b_{36}=100,\nonumber\\&& b_{18}=b_{34}=107,~b_{20}=b_{32}=111,~ b_{22}=b_{30}=112,\nonumber\\&& b_{24}=b_{26}=b_{28}=113.\nonumber\end{eqnarray}
In particular the Euler number of both moduli spaces is $1695.$ 
\end{thm}

\begin{flushleft}{\textbf{$\lozenge$ Computation for $[M(5,1)]$}}\end{flushleft}
According to Lemma \ref{ncss} we first stratify $M(5,1)$ into three strata defined as follows.
\begin{eqnarray} & M_1:=&\{[(E,f)]\in M(5,1)|E\simeq\mo_{\p^2}\oplus\mmon^{\oplus 4}\};\nonumber\\& M_2:=&\{[(E,f)]\in M(5,1)|E\simeq\mo_{\p^2}^{\oplus 2}\oplus\mmon^{\oplus 2}\oplus\mo_{\p^2}(-2)\};\nonumber\\& M_3:=&\{[(E,f)]\in M(5,1)|E\simeq\mone\oplus\mo_{\p^2}\oplus\mmon\oplus\mo_{\p^2}(-2)^{\oplus 2}\}.\nonumber
\end{eqnarray}

\begin{lemma}\label{scffi}A pair $(E,f)$ with $rank(E)=5$ and $deg(E)=-4$ is stable if and only if for any two direct summands $D',D''$ of $E$ such that $D'\simeq D''$ and $f(D'\otimes\mmon)\subset D''$,  we have $\mu(D')<\mu(E)$.
\end{lemma}
\begin{proof}See Appendix A. 
\end{proof}
For a pair $(E,f)\in M_3$, $f$ can be represented by the following matrix
\begin{equation}\label{romfi}\left(\begin{array}{ccccc}0&1&0&0&0\\0&0&1&0&0\\0&0&0&1&0\\a_1&0&0&0&b_1\\ a_2 &0&0&0&b_2\end{array}\right),
\end{equation}
where $b_i\in H^0(\mone)$ and $a_i\in H^0(\mo_{\p^2}(4))$.  The injectivity of $f$ implies that $det(f)=b_2a_1-b_1a_2\neq 0$.  Moreover by Lemma \ref{scffi} $(E,f)$ is stable if and only if $kb_1\neq k'b_2$ for any $(k,k')\in\bc^2-\{0\}$.  
\begin{lemma}\label{mcfimth}$[M_3]=[\p^2\times\p^{19}].$
\end{lemma}
\begin{proof}The proof is analogous to that of Lemma \ref{mcfmt}.  $M_3$ is isomorphic to a projective bundle over $Hilb^{[1]}(\p^2)$ with fibers isomorphic to $\p(H^0(I_1(5)))\simeq \p^{19}$. Hence the lemma.
\end{proof}
We stratify $M_2$ into two strata as follows.
\begin{eqnarray} & M_2^s:=&\{[(E,f)]\in M_2|f|_{\mo_{\p^2}^{\oplus 2}\otimes\mmon} is~surjective ~onto ~\mmon^{\oplus 2}\};\nonumber\\& M_2^c:=&M_2-M_2^s.\nonumber
\end{eqnarray}

For a pair $(E,f)\in M_2^s$, $f$ can be represented by the following matrix
\begin{equation}\label{romfit}\left(\begin{array}{ccccc}0&0&1&0&0\\0&0&0&1&0\\0&0&0&0&1\\b_1&b_2&0&0&0\\ a_1 &a_2&0&0&0\end{array}\right),
\end{equation}
where $b_i\in H^0(\mo_{\p^2}(2))$ and $a_i\in H^0(\mo_{\p^2}(3))$.  The injectivity of $f$ implies that $det(f)=b_1a_2-b_2a_1\neq 0$.  Moreover by Lemma \ref{scffi} $(E,f)$ is stable if and only if $kb_1\neq k'b_2$ for any $(k,k')\in\bc^2-\{0\}$.  
\begin{lemma}\label{mcfimt}$[M_2^s]=[Gr(2,6)\times\p^{16}-\p^2\times\p^2\times\p^2]$ with $Gr(2,6)$ the Grassmannian parametrizing 2-dimensional linear subspaces of $\bc^6.$
\end{lemma}
\begin{proof}We have the following diagram
\begin{equation}\label{fficd}\xymatrix{ 0\ar[r]&\mo_{\p^2}(-2)\ar[r]^{(b_1,b_2)}&\mo_{\p^2}^{\oplus 2}\ar[r]^{f_r}& R_f\ar[r]&0\\
&&\mo_{\p^2}(-3)\ar[u]^{(a_1,a_2)}\ar[ru]_{\omega_f:=f_r\circ(a_1,a_2)}&&}.\end{equation}
Since $kb_1\neq k'b_2,\forall(k,k')\in\bc^2-\{0\}$,  the isomorphism classes of $R_f$ 1-1 correspond to points  in $Gr(2,h^0(\mo_{\p^2}(2)))=Gr(2,6).$  Denote by $\mg$ the tautological bundle on $Gr(2,6)$.  Then on $Gr(2,6)\times\p^2$ we have the following exact sequence.
\begin{equation}\label{ufog}0\ra q^{*}\mo_{\p^2}(-2)\ra p^{*}\mg^{\vee}\ra\mathcal{R}\ra0,\end{equation}
with $p$ and $q$ the projections to $Gr(2,6)$ and $\p^2$ respectively.  $\mathcal{R}$ restricted to the fiber over $[(b_1,b_2)]\in Gr(2,6)$ is $R_f$.  Hence isomorphism classes of $(R_f,\omega_f)$ are parametrized by the projective bundle $\p(p_{*}(\mathcal{R}\otimes q^{*}\mo_{\p^2}(3)))$ over $Gr(2,6)$ with fibers isomorphic to $\p^{16}$.

However $M_2^s\subsetneq\p(p_{*}(\mathcal{R}\otimes q^{*}\mo_{\p^2}(3)))$ and the complement of $M_2^s$ in $\p(p_{*}(\mathcal{R}\otimes q^{*}\mo_{\p^2}(3)))$ consists of all $(R_f,\omega_f)$ such that the images of $\omega_f$ are contained in the torsions of $R_f$.

If $b_1$ is prime to $b_2$, then $R_f$ is torsion free.  If $b_1$ is not prime to $b_2$, then $R_f$ lies in the following exact sequence.
\begin{equation}\label{trfft}0\ra\mo_{H}(-1)\ra R_f\ra I_1(1)\ra 0,
\end{equation}
with $H$ a hyperplane in $\p^2.$  The closed subset $|H|\times Hilb^{[1]}(\p^2)\hookrightarrow Gr(2,6)$ parametrizes all the $R_f$ that are not torsion free.  

We write down the following diagram. 
\[\xymatrix{\p^2 & |H|\times Hilb^{[1]}(\p^2)\times\p^2\ar[l]_{q~~~~~~~~~~~}\ar[d]_{p_1}\ar[rd]^{p}\\ 
&\mathcal{C}_1\hookrightarrow |H|\times\p^2 & |H|\times Hilb^{[1]}(\p^2)}.\]
Those $(R_f,\omega_f)$ not in $M_2^s$ are parametrized by the projective bundle $\p(p_{*}(p_1^{*}\mo_{\mc_1}\otimes q^{*}\mo_{\p^2}(2)))$ over $|H|\times Hilb^{[1]}(\p^2)$ with fibers isomorphic to $\p(H^0(\mo_{H}(2)))\simeq \p^2$.

Hence we have $M_2^s\simeq \p(p_{*}(\mathcal{R}\otimes q^{*}\mo_{\p^2}(3)))-\p(p_{*}(p_1^{*}\mo_{\mc_1}\otimes q^{*}\mo_{\p^2}(2)))$, hence the lemma.
\end{proof}
  
For a pair $(E,f)\in M_2^c$, $f$ can be represented by the following matrix
\begin{equation}\label{romfitc}\left(\begin{array}{ccccc}b_1&b_2&0&0&0\\0&0&1&0&0\\a_1&a_2&0&b_3&0\\0&0&0&0&1\\ e_1 &e_2&0&a_3&0\end{array}\right),
\end{equation}
where $b_i\in H^0(\mo_{\p^2}(1))$, $a_i\in H^0(\mo_{\p^2}(2))$ and $e_i\in H^0(\mo_{\p^2}(3))$. $det(f)\neq 0$.  Lemma \ref{scffi} implies that $(E,f)$ is stable if and only if $kb_1\neq k'b_2,\forall (k,k')\in\bc^2-\{0\}$ and $k''a_3\neq b\cdot b_3,\forall (k'',b)\in \bc\times H^0(\mone)-\{(0,0)\}$.  
\begin{lemma}\label{mcfimtt}$[M_2^c]=[Hilb^{[1]}(\p^2)\times Hilb^{[2]}(\p^2)\times\p^{17}-Hilb^{[1]}(\p^2)\times\p^1\times\p^1]$.\end{lemma}  
\begin{proof}We first write down the following two exact sequences.
\begin{equation}\label{fcsmc}\xymatrix{ 0\ar[r]&\mo_{\p^2}(-1)\ar[r]^{~~(b_1,b_2)}&\mo_{\p^2}^{\oplus 2}\ar[r]^{f_r}& R_f\ar[r]&0}\end{equation}
\begin{equation}\label{scsmc}\xymatrix{ 0\ar[r]&\mo_{\p^2}(-1)\ar[r]^{(a_3,b_3)~~~}&\mone\oplus\mo_{\p^2}\ar[r]^{~~~~~~f_s}& S_f\ar[r]&0}\end{equation}

Because of the stability condition, we see that both $R_f$ and $S_f$ are torsion free and hence $R_f\simeq I_1(1)$ and $S_f\simeq I_2(2)$.  On the other hand, any $I_1(1)$ ($I_2(2)$) can be put in the sequence (\ref{fcsmc}) ((\ref{scsmc})).  

We write down a commutative diagram as follows.
\begin{equation}\label{fmcbcd}\xymatrix@C=1.5cm@R=0.7cm{&0\ar[d] &0\ar[d] &0\ar[d]\\
0\ar[r]&\mo_{\p^2}(-2)\ar[r]^{(a_3,b_3)\otimes id_{\mmon}~~~}\ar[d]_{(b_1,b_2)\otimes id_{\mmon}}&\mo_{\p^2}\oplus \mo_{\p^2}(-1)\ar[d]^{(b_1,b_2)\otimes id_{\mone}\oplus (b_1,b_2)}\ar[r]^{~~~~f_s\otimes id_{\mmon}}&S_f(-1)\ar[r]\ar[d]^{id_{S_f}\otimes(b_1,b_2)}&0\\
0\ar[r]&\mo_{\p^2}(-1)^{\oplus 2}\ar[r]_{(a_3,b_3)^{\oplus 2}~~~~~}\ar[d]_{f_r\otimes id_{\mmon}}&(\mone\oplus\mo_{\p^2})^{\oplus 2}\ar[d]^{f_r\otimes id_{\mone}\oplus f_r}\ar[r]_{~~~~~~~~f_s^{\oplus 2}}& S_f^{\oplus 2}\ar[r]\ar[d]^{id_{S_f}\otimes f_r}&0\\
0\ar[r]&R_f\otimes \mo_{\p^2}(-1)\ar[r]_{id_{R_f}\otimes(a_3,b_3)}\ar[d] &R_f\oplus R_f(1)\ar[d]\ar[r]_{id_{R_f}\otimes f_s}&R_f\otimes S_f\ar[r]\ar[d]&0\\
&0 &0 &0}
\end{equation}
We have another commutative diagram
\begin{equation}\label{fmcscd}\xymatrix@C=1.7cm{\mo_{\p^2}(-2)\ar[r]^{(e_1,a_1)\oplus (e_2,a_2)~~~~~~}&(\mone\oplus\mo_{\p^2})^{\oplus 2}\ar[d]_{f_r\otimes id_{\mone}\oplus f_r}\ar[r]^{~~~~~~~~f_s^{\oplus 2}}& S_f^{\oplus 2}\ar[d]^{id_{S_f}\otimes f_r}\\
 &R_f\oplus R_f(1)\ar[r]_{id_{R_f}\otimes f_s} &R_f\otimes S_f.} 
\end{equation}
Isomorphism classes of $(E,f)\in M_2^c$ are parametrized by $(R_f,S_f,\omega_f)$ with $\omega_f: \mo_{\p^2}(-2)\ra R_f\otimes S_f$ the composed map in (\ref{fmcscd}).  We write down the following diagram. 
\[\xymatrix{\p^2 & Hilb^{[1]}(\p^2)\times Hilb^{[2]}(\p^2)\times\p^2\ar[l]_{q~~~~~~~~~~~}\ar[d]_{p}\ar[rd]^{p_2}\ar[ld]_{p_1}\\ 
Hilb^{[1]}(\p^2)\times\p^2&Hilb^{[1]}(\p^2)\times Hilb^{[2]}(\p^2) & Hilb^{[2]}(\p^2)\times\p^2}.\]
Denote by $\mathcal{I}_1$ ($\mathcal{I}_2$) the universal family of ideal sheaves on $Hilb^{[1]}(\p^2)\times\p^2$ ($Hilb^{[2]}(\p^2)\times\p^2$).  Isomorphism classes of $(R_f,S_f,\omega_f)$ are parametrized by the projective bundle $\p(p_{*}(p_1^{*}\mathcal{I}_{1}\otimes p_2^{*}\mathcal{I}_{2}\otimes q^{*}\mo_{\p^2}(5)))$ over $Hilb^{[1]}(\p^2)\times Hilb^{[2]}(\p^2)$ with fibers isomorphic to $\p(H^0(I_1\otimes I_2\otimes \mo_{\p^2}(5)))\simeq \p^{17}$.

There are still points in $\p(p_{*}(p_1^{*}\mathcal{I}_{1}\otimes p_2^{*}\mathcal{I}_{2}\otimes q^{*}\mo_{\p^2}(5)))$ that we must exclude.  They are points $(R_f,S_f,\omega_f)$ such that the images of $\omega_f$ are contained in the torsions of $R_f\otimes S_f$. 

We write down the following exact sequence.
\begin{equation}\label{totis}0\ra I_1\ra\mo_{\p^2}\ra \mo_{x}\ra0,
\end{equation}
with $\mo_{x}$ the skyscraper sheaf supported at a single point $x$.  Tensor (\ref{totis}) by $I_2$, and we get
\[0\ra Tor^1(\mo_{x},I_2)\ra I_1\otimes I_2\ra I_2\ra I_2\otimes\mo_x\ra 0.\]
We see that the torsion of $I_1\otimes I_2$ is isomorphic to $Tor^1(\mo_{x},I_2)$.  Tensor (\ref{scsmc}) by $\mo_{x}$ and we get
\[\xymatrix@C=1cm{0\ar[r] &Tor^1(\mo_x,I_2(2))\ar[r] &\mo_x\ar[r]^{(a_3,b_3)~~}&\mo_x^{\oplus 2}\ar[r] &I_2(2)\otimes \mo_x\ar[r]&0.}\]
Hence we see that the torsion of $R_f\otimes S_f$ is either zero or isomorphic to $\mo_x$.  The later implies that $R_f\simeq I_{\{x\}}(1)$ and $Tor^1(\mo_x,I_2)\neq 0\Leftrightarrow (a_3,b_3)|_x=0.$  We then want to parametrize all $(I_1,I_2)$ such that $I_1\otimes I_2$ contain torsion.  We first write down the following diagram
\begin{equation}\label{mstotis}\xymatrix{\p^2&\p^2\times Hilb^{[1]}(\p^2)\times |H|\ar[l]_{q~~~~~}\ar[ld]_{q_1}\ar[d]_{p_3}\\ \mc_1\hookrightarrow \p^2\times |H| & Hilb^{[1]}(\p^2)\times |H| \hookleftarrow \p(\mathcal{V}^1),}
\end{equation} 
where $\mathcal{V}^1$ is the rank 2 vector bundle on $Hilb^{[1]}(\p^2)$ defined as $s_{*}(\mathcal{I}_1\otimes t^{*}\mo_{\p^2}(1))$ with $s$ and $t$ the projection from $Hilb^{[1]}(\p^2)\times \p^2$ to $Hilb^{[1]}(\p^2)$ and $\p^2$ respectively.  Let $\mathcal{Z}$ be defined by the following Cartesian diagram
\begin{equation}\label{cadffi}\xymatrix{\mathcal{Z}\ar[r]\ar[d] &\p(p_{3*}(q_1^{*}(\mo_{\mc_1})\otimes q^{*}\mo_{\p^2}(1)))\ar[d]\\ \p(\mathcal{V}^1)\ar[r] &Hilb^{[1]}(\p^2)\times |H|. }
\end{equation}
$\p(\mathcal{V}^1)$ parametrizes all the pairs $([x],[C])\in Hilb^{[1]}(\p^2)\times |H|$ such that $x\in C$, and $\mathcal{Z}$ parametrizes $([x_1],[(C,x_2)])\in Hilb^{[1]}(\p^2)\times \mc_1$ such that $x_1\in C$.  Define $\imath:\mathcal{Z}\ra Hilb^{[1]}(\p^2)\times Hilb^{[2]}(\p^2)$ such that $\imath(([x_1],[(C,x_2)]))=([x_1],[x_1,x_2,C])$.  It is easy to see that $\imath$ is an embedding with its image exactly the set of points $(I_1,I_2)$ such that $I_1\otimes I_2$ have torsion.   

$p_{*}(p_1^{*}\mathcal{I}_1\otimes p^{*}_1\mathcal{I}_2 \otimes p^{*}\mo_{\mathcal{Z}})\simeq p_{*}(p_1^{*}\mathcal{I}_1\otimes p^{*}_1\mathcal{I}_2) \otimes \mo_{\mathcal{Z}}$ by the flatness of $p$. $p_1^{*}\mathcal{I}_1\otimes p^{*}_1\mathcal{I}_2 \otimes p^{*}\mo_{\mathcal{Z}}$ contains $p_1^{*}\mo_{\mathcal{Z}_1}$ as its torsion where $\mathcal{Z}_1$ is the universal subscheme in $Hilb^{[1]}(\p^2)\times \p^2$.  Hence we can embed $\mathcal{Z}$ into $\p(p_{*}(p_1^{*}\mathcal{I}_{1}\otimes p_2^{*}\mathcal{I}_{2}\otimes q^{*}\mo_{\p^2}(5)))$ by taking the non-zero constant section of $p_1^{*}\mo_{\mathcal{Z}_1}.$   

Hence we have $M_2^c\simeq \p(p_{*}(p_1^{*}\mathcal{I}_{1}\otimes p_2^{*}\mathcal{I}_{2}\otimes q^{*}\mo_{\p^2}(5)))-\mathcal{Z}$.  The lemma follows since $[\mathcal{Z}]=[Hilb^{[1]}(\p^2)\times\p^1\times\p^1].$
\end{proof}

Finally let $(E,f) \in M_1-W^5$.  Rewrite (\ref{boscd}) for $d=5$ as follows
\begin{equation}\label{tqficd}\xymatrix{ 0\ar[r]&\mmon^{\oplus 3}\ar[r]^{~~f_{B^t}}&\mo_{\p^2}^{\oplus 4}\ar[r]^{f_q}& Q_f\ar[r]&0\\
&&\mo_{\p^2}(-2)\ar[u]^{f_{A^t}}\ar[ru]_{\sigma_f=f_q\circ f_{A^t}}&&}.\end{equation}
Notice that the torsion $T_f$ of $Q_f$ contains neither $\mo_{H}$ nor $\mo_{2H}(x)$ as a subsheaf with $x$ a single point on the curve, otherwise we will have a diagram similar to diagram (\ref{ssff}) which contradicts Condition \ref{scodm}.  Also $H^0(T_f\otimes\mmon)=0.$  Hence we see that if $c_1(T_f)=2H$, then $\chi(T_f)=1$ and $Q_f^{tf}\simeq \mo_{\p^2}(1)$.  Moreover since $T_f$ does not contain $\mo_{H}(n)$ with $n\geq0$, $T_f$ is stable and hence by Theorem \ref{mtott} there is only one stable sheaf for each curve in $|2H|$.  Hence $T_f\simeq \mo_{2H}$.

We have the following commutative diagram.
 \begin{equation}\label{tqcdffi}\xymatrix@C=1cm@R=0.5cm{& &0\ar[d] &0\ar[d]\\
0\ar[r]&\mmon^{\oplus 3}\ar[r]^{~~~\imath}\ar[d]^{\simeq}&K \ar[d]_{j}\ar[r]&T_f\ar[r]\ar[d]&0\\
0\ar[r]&\mmon^{\oplus 3}\ar[r]^{~~~f_{B^t}}&\mo_{\p^2}^{\oplus 4}\ar[d]_{f_{tq}}\ar[r]^{f_q}& Q_f\ar[r]\ar[d]&0\\
& &Q^{tf}_f\ar[r]^{\simeq}\ar[d] &Q_f^{tf}\ar[d]\\& &0 &0}
\end{equation} 

We stratify $M_1-W^5$ into three strata as follows.
\begin{eqnarray} & \Pi_1:=&\{[(E,f)]\in M_1-W^5|T_f\simeq \mo_{2H},Q^{tf}_f\simeq\mone\};\nonumber\\&\Pi_2:=&\{[(E,f)]\in M_1-W^5|T_f\simeq \mo_{H}(-1),Q_f^{tf}\simeq I_2(2)\};\nonumber\\&\Pi_3:=&\{[(E,f)]\in M_1-W^5|T_f\simeq \mo_{H}(-2),Q_f^{tf}\simeq I_1(2)\}.\nonumber
\end{eqnarray}
A priori there is the fourth possibility that $T_f\simeq\mo_{H}(-3),Q_f^{tf}\simeq\mo_{\p^2}(2)$, we will explain why this case is excluded later in the computation for $[\Pi_3]$.

\begin{lemma}\label{pione}$[\Pi_1]=[\p^5\times\p^{14}-\p^5\times\p^4]$.
\end{lemma}
\begin{proof}Notice that Ext$^i(\mone,\mo_{2H})=0$ for all $i\neq 1$ and Ext$^1(\mone,\mo_{2H})\simeq\bc$, and the proof is analogous to that of Lemma \ref{lpff}.
\end{proof} 

Let $(E,f)\in \Pi_2$.  Since $Q_f^{tf}\simeq I_2(2)$, we have the following exact sequence
\begin{equation}\label{lfrit}0\ra \mmon\stackrel{(a,b)}{\longrightarrow}\mone\oplus\mo_{\p^2}\stackrel{g}{\longrightarrow}I_2(2)\ra0,
\end{equation}
with $b\in H^0(\mone)$, $a\in H^0(\mo_{\p^2}(2))$ and $b$ prime to $a$.  Take Hom$(-,\mo_{H}(-1))$ on (\ref{lfrit}) and we get 
\begin{equation}\label{lfritgh}\small{0\ra \textrm{Hom}(\mmon,\mo_{H}(-1))\ra\textrm{Ext}^1(I_2(2),\mo_{H}(-1))\stackrel{\tilde{g}}{\rightarrow} \textrm{Ext}^1(\mone\oplus\mo_{\p^2},\mo_{H}(-1))\ra0},
\end{equation}
\begin{lemma}\label{qfpio}$Q_f$ with $T_f\simeq \mo_H(-1)$ and $Q_f^{tf}\simeq I_2(2)$ lies in diagram (\ref{tqcdffi}) if and only if the image of the following exact sequence via $\tilde{g}$ is not zero
\begin{equation}\label{tqffl}0\ra \mo_H(-1)\ra Q_f\ra I_2(2)\ra0,\end{equation}
i.e. (\ref{tqffl}) is not contained in the image of \textrm{Hom}$(\mmon,\mo_H(-1)).$
\end{lemma}
\begin{proof}The map $\tilde{g}$ gives the following commutative diagram
\begin{equation}\label{pitcd}\xymatrix@C=0.7cm@R=0.5cm{& &0\ar[d] &0\ar[d]\\
&&\mmon\ar[d]\ar[r]^{\simeq}&\mmon\ar[d]^{(a,b)}&\\
0\ar[r]&\mo_H(-1)\ar[r]\ar[d]^{\simeq}&\widetilde{Q_f}\ar[d]^{\delta}\ar[r]^{}& \mone\oplus\mo_{\p^2}\ar[r]\ar[d]^g&0~~(*)\\
0\ar[r]&\mo_H(-1)\ar[r]&Q_f\ar[r]^{}\ar[d] &I_2(2)\ar[d]\ar[r]&0~~(**)\\& &0 &0}
\end{equation}
where the sequence $(*)$ is the image of the sequence $(**)$ via $\tilde{g}$ and $\widetilde{Q_f}$ is the Cartesian product of $Q_f$ and $\mone\oplus\mo_{\p^2}$ over $I_2(2)$.  

From (\ref{pitcd}) we see that Hom$(\mo_{\p^2},\widetilde{Q_f})\simeq$ Hom$(\mo_{\p^2},\mone\oplus\mo_{\p^2})$ and Hom$(\mo_{\p^2},Q_f)\simeq$ Hom$(\mo_{\p^2},I_2(2))$.  Moreover the map $f_{tq}$ in (\ref{tqcdffi}) factors through a surjective map $s:\mo_{\p^2}^{\oplus 4}\ra\mone\oplus\mo_{\p^2}$ since Hom$(\mo_{\p^2}^{\oplus 4},\mone\oplus\mo_{\p^2})\twoheadrightarrow$Hom$(\mo_{\p^2}^{\oplus 4},I_2(2))$, and $s$ lifts to a map $\tilde{s}:\mo_{\p^2}^{\oplus 4}\ra \widetilde{Q_f}$ such that $f_q=\delta\circ\tilde{s}$. 

If the sequence $(*)$ in (\ref{pitcd}) splits, then $\widetilde{Q_f}\simeq \mone\oplus\mo_{\p^2}\oplus \mo_H(-1)$ and $\delta\circ\tilde{s}$ can not be surjective.  Hence $(*)$ does not split.  

On the other hand, Ext$^1(\mone\oplus\mo_{\p^2},\mo_{H}(-1))\simeq \bc$, hence $\widetilde{Q_f}$ is unique up to isomorphism if $(*)$ does not split.   We see that in this case $\widetilde{Q_f}\simeq\mo_{\p^2}\oplus Q_f^1$ with $Q_f^1$ lying in the following non-splitting sequence
\[0\ra\mo_H(-1)\ra Q_f^1\ra\mone\ra0.\]
By Lemma \ref{tqns} $Q_f^1$ lies in the following sequence
\[0\ra\mmon^{\oplus 2}\ra\mo_{\p^2}^{\oplus 3}\stackrel{f_{q_1}}{\longrightarrow}  Q_f^1\ra0.\]
We then have the following commutative diagram
\begin{equation}\label{pitccd}\xymatrix@C=0.7cm@R=0.5cm{&0\ar[d] &0\ar[d]&\\
&\mmon^{\oplus 2}\ar[d]\ar[r]^{\simeq}&\mmon^{\oplus 2}\ar[d]^{(a,b)}&&\\
0\ar[r]&\mmon^{\oplus 3}\ar[r]\ar[d]^{}&\mo_{\p^2}^{\oplus 4}\ar[d]^{f_{q_1}\oplus id_{\mo_{\p^2}}}\ar[r]^{f_q}& Q_f\ar[r]\ar[d]^{\simeq}&0\\
0\ar[r]&\mo_{\p^2}(-1)\ar[d]\ar[r]&Q_f^1\oplus \mo_{\p^2}\ar[r]^{}\ar[d] &Q_f\ar[r]&0\\&0 &0 &}
\end{equation}
Hence $Q_f$ lies in (\ref{pitcd}) and hence the lemma.
\end{proof}
\begin{lemma}\label{mtopit}$[\Pi_2]=[Hilb^{[2]}(\p^2)\times |H|\times (\p^1-1)\times (\p^{14}-\p^1)].$
\end{lemma}
\begin{proof}Lemma \ref{qfpio} implies that for fixed $\mo_H(-1)$ and $I_2(2)$, isomorphism classes of $Q_f$ are parametrized by $\p(\textrm{Ext}^1(I_2(2),\mo_{H}(-1)))-\p(\textrm{Hom}(\mmon,\mo_{H}(-1)))$.  Hence isomorphism classes of all $Q_f$ are parametrized by the following scheme
\[\p(\mathcal{E}xt^1_p(\mathcal{I}_2\otimes q^{*}\mo_{\p^2}(2),\mo_{\mc_1}\otimes q^{*}\mo_{\p^2}(-1)))-Hilb^{[2]}(\p^2)\times |H|;\]
where $p$ and $q$ are projections from $\p^2\times Hilb^{[2]}(\p^2)\times |H|$ to $Hilb^{[2]}(\p^2)\times |H|$ and $\p^2$ respectively, $\mathcal{I}_2$ and $\mc_1$ are the pull back of the universal ideal sheaf and the universal curve to $\p^2\times Hilb^{[2]}(\p^2)\times |H|$ from $\p^2\times Hilb^{[2]}(\p^2)$ and $\p^2\times |H|$ respectively.  Notice that we embed $Hilb^{[2]}(\p^2)\times |H|$ into $\p(\mathcal{E}xt^1_p(\mathcal{I}_2\otimes q^{*}\mo_{\p^2}(2),\mo_{\mc_1}\otimes q^{*}\mo_{\p^2}(-1)))$ by taking the nonzero constant section of the line bundle 
$\mathcal{H}om_p(q^{*}\mmon,\mo_{\mc_1}\otimes q^{*}\mmon)\simeq p_{*}\mo_{\mc_1}\simeq \mo_{Hilb^{[2]}(\p^2)\times |H|}$.

Analogously the space parametrizing $(Q_f,\sigma_f)$ is the difference of two projective bundles with fibers isomorphic to $\p(H^0(Q_f(2)))\simeq \p^{14}$ and $\p(H^0(\mo_H(1)))\simeq\p^1$ respectively over the space parametrizing $Q_f$.  Hence the lemma.
\end{proof}

Now we do the computation for $[\Pi_3]$ and we will also explain why the case that $T_f\simeq \mo_{H}(-3),Q^{tf}_f\simeq \mo_{\p^2}(2)$ is not included.  Notice that the map $f_{tq}$ in (\ref{tqcdffi}) is not surjective on global sections.  We first write down the following commutative diagram.
\begin{equation}\label{cdpth}\xymatrix@C=0.7cm@R=0.5cm{&0\ar[d] &0\ar[d]&\\
0\ar[r]&K\ar[d]\ar[r]^{j}&\mo_{\p^2}^{\oplus 4}\ar[d]^{}\ar[r]^{f_{tq}}&Q^{tf}_f\ar[r]\ar[d]^{\simeq}&0\\
0\ar[r]&G\ar[r]\ar[d]_{\tau}&\mo_{\p^2}^{\oplus 4+n}\ar[d]^{\tilde{\tau}}\ar[r]^{g}& Q^{tf}_f\ar[r]&0\\
&\mo_{\p^2}^{\oplus n}\ar[d]\ar[r]^{\simeq}&\mo_{\p^2}^{\oplus n}\ar[d] &&\\&0 &0 &}
\end{equation}
where $n=1$ if $Q_f^{tf}\simeq I_1(2)$ and $n=2$ if $Q_f^{tf}\simeq \mo_{\p^2}(2)$.  

From (\ref{cdpth}) we see that $H^i(G(1-i))=0$ for $i>0$, hence by Castelnuovo-Mumford regularity $G(1)$ is globally generated.  Therefore the map $\tau\otimes id_{\mone}:G(1)\ra\mone^{\oplus n}$ must be surjective on global sections, since otherwise $\tau$ is not surjective.  Hence $h^0(K(1))=h^0(G(1))-nh^0(\mone)$.  So if $n=2,Q_f^{tf}\simeq \mo_{\p^2}(2)$, then we have $h^0(K(1))=2$ which implies that $K$ can not contain $\mmon^{\oplus 3}$ as a subsheaf.  Hence we only have $n=1$ and $Q_f^{tf}\simeq I_1(2).$

$T_f\simeq \mo_H(-2)$ hence Hom$(\mmon, T_f)=0$, the inclusion $\imath$ in (\ref{tqcdffi}) is unique up to isomorphisms of $\mmon^{\oplus 3}$ for a fixed $K$.  Hence $f_{B^t}$ is determined by the inclusion $j$ and hence is determined by $f_{tq}$.  Parametrizing $f_{B^t}$ is equivalent to parametrizing the surjective map $f_{tq}$, hence equivalent to parametrizing $\tilde{\tau}$.  We first assume $Q_f^{tf}\simeq I_{[0,0,1]}(2)$,  then $g$ can be represented by a $1\times 5$ matrix $(x_0^2,x_0x_1,x_0x_2,x_1x_2,x_1^2)$.  $\tilde{\tau}$ can be represented by $\underline{h}:=(h_0,h_1,h_2,h_3,h_4)$ with $h_i\in\bc$.  We want to parametrize the class of $\underline{h}$ modulo scalars.

The sheaf $G$ can be generated by 6 generators $<\e_0,\e_1,\e_2,\ee_0,\ee_1,\ee_2>$ in $H^0(G(1))$ with two syzygies $(x_0\e_0+x_1\e_1-x_2\e_2=0,x_0\ee_0+x_1\ee_1-x_2\ee_2=0).$

The map $\tau$ is determined by $\tau^0:=h^0(\tau\otimes id_{\mone}):H^0(G(1))\ra H^0(\mone)$, and also $\tau$ is induced by $\tilde{\tau}$.  Hence $\tau^0$ is determined by $\underline{h}$ and we can write down explicitly the images of $\e_i$ and $\ee_i$ as follows
\begin{eqnarray}
&&\tau^0(\e_0)=h_2x_1-h_1x_2;~~ \tau^0(\e_1)=h_0x_2-h_2x_0;~~\tau^0(\e_2)=h_0x_1-h_1x_0\nonumber\\&& \tau^0(\ee_0)=h_3x_1-h_4x_2;~~\tau^0(\ee_1)=h_1x_2-h_3x_0;~~\tau^0(\ee_2)=h_1x_1-h_4x_0.\nonumber\end{eqnarray}

We can get (\ref{cdpth}) if and only if $\tau^0$ is surjective.  In other words, the following $3\times 6$ matrix has rank 3.
\[Mat_{\tau}:=\left(\begin{array}{cccccc}0&-h_2&-h_1&0&-h_3&-h_4\\h_2&0&h_0&h_3&0&h_1\\-h_1&h_0&0&-h_4&h_1&0\end{array}\right)\]
By direct computation we see that 
\[rank (Mat_{\tau})<3\Leftrightarrow h_1h_2-h_0h_3=h_1^2-h_0h_4=h_1h_3-h_2h_1=0.\]  Hence we know that $f_{tq}$ are parametrized by $P_{\tau}:=\p^4-\{h_1h_2-h_0h_3=h_1^2-h_0h_4=h_1h_3-h_2h_1=0\}.$  

One can easily compute that $[P_{\tau}]=[\p^4-(\p^1+\mathbb{A}^2+\mathbb{A}^1)]$.  Moreover we can cover $Hilb^{[1]}(\p^2)$ by finitely many Zariski open subsets $U_i$ such that $(I_1(2),\tilde{\tau})$ with $[I_1]\in U_i$ are parametrized by $U_i\times P_{\tau}$.  For example, we can take $U_i$ such that $p_{*}(\mathcal{I}_{1}\otimes q^{*}\mo_{\p^2}(2))|_{U_i}\simeq \mo_{U_i}^{\oplus 5}$, where $p$ and $q$ are the projections from $Hilb^{[1]}(\p^2)\times\p^2$ to $Hilb^{[1]}(\p^2)$ and $\p^2$ respectively and $\mathcal{I}_1$ the universal ideal sheaf.

$(I_1(2),\tilde{\tau})$ determines $Q_f$ and analogously we know that $(Q_f,\sigma_f)$ are parametrized by a difference of two projective bundles over the space parametrizing $Q_f$.  Hence we have the following lemma as a direct consequence.
\begin{lemma}\label{mthopit}$[\Pi_3]=[Hilb^{[1]}(\p^2)\times (\p^4-(\p^1+\mathbb{A}^2+\mathbb{A}^1))\times (\p^{14}-1)]$.
\end{lemma}  

We have already known that $[W^5]=\p^{14}\times [Hilb^{[6]}(\p^2)-\Omega_{2}^{[6]}]$.  The proof of the following lemma is postponed to the appendix.
\begin{lemma}\label{ctome}$[\Omega_{2}^{[6]}]=\bl^{11}+3\bl^{10}+8\bl^{9}+18\bl^{8}+30\bl^{7}+39\bl^{6}+38\bl^{5}+28\bl^{4}+15\bl^{3}+6\bl^{2}+2\bl^{1}+1$.
\end{lemma}
\begin{proof}See Appendix B.
\end{proof}
\begin{proof}[Proof of Theorem \ref{mtfi} for $M(5,1)$] We have 
\[[M(5,1)]=[M_3]+[M_2^s]+[M_2^c]+\displaystyle{\sum_{i=1}^3}[\Pi_i]+([Hilb^{[6]}(\p^2)]-[\Omega_2^{[6]}])\times [\p^{14}].\] 
Combine Lemma \ref{mcfimth}, Lemma \ref{mcfimt}, Lemma \ref{mcfimtt}, Lemma \ref{pione}, Lemma \ref{mtopit}, Lemma \ref{mthopit} and Lemma \ref{ctome}, and we get the result by direct computation.
\end{proof}

\begin{flushleft}{\textbf{$\lozenge$ Computation for $[M(5,2)]$}}\end{flushleft}

We stratify $M(5,2)$ into three strata defined as follows.
\begin{eqnarray} & M_2:=&\{[(E,f)]\in M(5,1)|E\simeq\mo_{\p^2}^{\oplus 2}\oplus\mmon^{\oplus 3}\};\nonumber\\& M_3:=&\{[(E,f)]\in M(5,1)|E\simeq\mo_{\p^2}^{\oplus 3}\oplus\mmon\oplus\mo_{\p^2}(-2)\};\nonumber\\& M_3':=&\{[(E,f)]\in M(5,1)|E\simeq\mone\oplus\mo_{\p^2}\oplus\mmon^{\oplus 2}\oplus\mo_{\p^2}(-2)\}.\nonumber
\end{eqnarray}
Here we use notation $M_3'$ instead of $M_4$ because we want to specify the lower index of the subspace to be $h^0(F)$ with $F$ any sheaf in it. 
\begin{lemma}\label{scffiv}A pair $(E,f)$ with $rank(E)=5$ and $deg(E)=-3$ is stable if and only if for any two direct summands $D',D''$ of $E$ such that $D'\simeq D''$ and $f(D'\otimes\mmon)\subset D''$,  we have $\mu(D')<\mu(E)$.
\end{lemma}
\begin{proof}See Appendix A. 
\end{proof}
For a pair $(E,f)\in M'_3$, $f$ can be represented by the following matrix
\begin{equation}\label{romfiv}\left(\begin{array}{ccccc}0&1&0&0&0\\0&0&1&0&0\\0&0&0&0&1\\d&0&0&b&0\\  c&0&0&a&0\end{array}\right),
\end{equation}
where $b\in H^0(\mone)$, $a\in H^0(\mo_{\p^2}(2))$, $c\in H^0(\mo_{\p^2}(4))$ and $d\in H^0(\mo_{\p^2}(3))$.  $det(f)=ad-bc\neq 0$ and by Lemma \ref{scffi} $(E,f)$ is stable if and only if $b$ is prime to $a$.  
\begin{lemma}\label{mcfimthp}$[M'_3]=[Hilb^{[2]}(\p^2)\times\p^{18}].$
\end{lemma}
\begin{proof}We have the following exact sequence
\begin{equation}\label{eslsix}0\ra\mmon\stackrel{(a,b)}{\longrightarrow}\mone\oplus\mo_{\p^2}\ra I_2(2)\ra0, \end{equation}
and for every $[I_2]\in Hilb^{[2]}(\p^2)$, $I_2(2)$ lies in (\ref{eslsix}).  Hence analogous to Lemma \ref{mcfmt}, $M'_3$ is isomorphic to a projective bundle over $Hilb^{[2]}(\p^2)$ with fibers isomorphic to $\p(H^0(I_2(5)))\simeq \p^{18}$. Hence the lemma.
\end{proof}

For a pair $(E,f)\in M_3$, $f$ can be represented by the following matrix
\[\left(\begin{array}{ccc}B&0&0\\ 0&1 &0\\ 0 &0&1\\ A&0&0 \end{array}\right),\]
where $A$ is a $1\times 3$ matrix with entries in $H^0(\mo_{\p^2}(3))$ and $B$ a $2\times 3$ matrix with entries in $H^0(\mone)$.  The parametrizing space of $B$ is of class $[Hilb^{[3]}(\p^2)-|H|\times\p^3+|H|]$ by Lemma \ref{doof} and Lemma \ref{icqff}.  We have the following lemma.
\begin{lemma}\label{mcfivmth}$[M_3]=[(Hilb^{[3]}(\p^2)-|H|\times\p^3+|H|)\times\p^{17}-|H|\times\p^2].$
\end{lemma}
\begin{proof}$M_3$ is the union of a projective bundle over $Hilb^{[3]}(\p^2)-\Omega_1^{[3]}$ with fiber isomorphic to $\p(H^0(I_3(5)))\simeq\p^{17}$ and a difference of two projective bundles over $|H|$ with fibers isomorphic to $\p^{17}$ and $\p(H^0(\mo_H(2)))\simeq \p^2$ respectively.  Hence the lemma.
\end{proof}

We stratify $M_2$ into two strata as follows.
\begin{eqnarray} & M_2^s:=\{[(E,f)]\in M_2|f_{rs}:\mmon^{\oplus 2}\stackrel{f|_{\mmon^{\oplus 2}}}{\longrightarrow}E\twoheadrightarrow\mmon^{\oplus 3}~is ~injective\};&\nonumber\\& M_2^c:=M_2-M_2^s.\qquad\qquad\qquad\qquad\qquad\qquad\qquad\qquad\qquad\qquad~~~~~~~~~~~~~~~~~~~~~&\nonumber
\end{eqnarray}

For a pair $(E,f)\in M_2^s$, $f$ can be represented by the following matrix
\[\left(\begin{array}{cccc}0&1&0&0\\ 0&0&1 &0\\ A&0&0&B \end{array}\right),\]
where $A$ is a $3\times 2$ matrix with entries in $H^0(\mo_{\p^2}(2))$ and $B$ a $3\times 1$ matrix with entries in $H^0(\mone)$. 

We stratify $M_2^s$ into two strata as follows.
\begin{eqnarray} & \Xi_1:=\{[(E,f)]\in M_2^s|B\simeq (x_0,x_1,x_2)^t\};& \Xi_2:=M_2^s-\Xi_1.\nonumber
\end{eqnarray}
If $B\simeq (x_0,x_1,x_2)^t$, then $(E,f)$ always satisfies the stability condition.  We have the following diagram
\begin{equation}\label{cdxio}\xymatrix{ 0\ar[r]&\mmon\ar[r]^{~~(x_0,x_1,x_2)}&\mo_{\p^2}^{\oplus 3}\ar[r]^{f_0}& E_0\ar[r]&0\\
&&\mo_{\p^2}(-2)^{\oplus 2}\ar[u]^{f_{A^t}}\ar[ru]_{\xi_f:=f_0\circ f_{A^t}}&&},\end{equation}
with $E_0$ a rank $2$ bundle which is the dual of the kernel of the surjective map $\mo_{\p^2}^{\oplus 3}\stackrel{(x_0,x_1,x_2)^t}{\longrightarrow}\mone.$   Isomorphism classes of $\xi_f$ are parametrized by $Gr(2,15)$ since $h^0(E_0(2))=15$.  Moreover $det(f)\neq 0\Leftrightarrow$ the image of $\xi_f$ is a rank two subsheaf of $E_0\Leftrightarrow$ Im$(\xi_f)$ is not contained in a rank one subsheaf of $E_0$.

Assume Im$(\xi_f)$ is contained in a rank one subsheaf $E_1\subsetneq E_0$.  Since $E_0$ is locally free, we ask $E_1$ to be a line bundle.  Hence either $E_1\simeq \mo_{\p^2}$ or $E_1\simeq \mmon$.  Notice that for any $n$ a map $\mo_{\p^2}(n)\ra E_0$ always factors through map $f_0$ in (\ref{cdxio}).  

Since Hom$(\mo_{\p^2},E_0)\simeq\textrm{Hom}(\mo_{\p^2},\mo_{\p^2}^{\oplus 3})$, all inclusions $i:\mo_{\p^2}\hookrightarrow E_0$ are parametrized by Hom$(\mo_{\p^2},\mo_{\p^2}^{\oplus 3})-\{0\}$.  Moreover $\forall i,i'\in H^0(\mo_{\p^2},\mo_{\p^2}^{\oplus 3})-\{0\},i\not\equiv i'$,  Im$(i)\cap\textrm{Im}(i')=\emptyset$.  Hence all $\xi_f$ such that Im$(\xi_f)$ are contained in $\mo_{\p^2}\simeq E_1\subsetneq E_0$ are parametrized by $Gr(2,h^0(\mo_{\p^2}(2)))\times \p(\textrm{Hom}(\mo_{\p^2},\mo_{\p^2}^{\oplus 3}))\simeq Gr(2,6)\times \p^2$.

Let $\imath\in\textrm{Hom}(\mo_{\p^2}(-1),\mo_{\p^2})-\{0\}$.  All inclusions $j:\mmon\hookrightarrow E_0$ that do not factor through $\imath:\mmon\hookrightarrow\mo_{\p^2}$ are parametrized by Hom$(\mo_{\p^2}(-1),E_0)-\tilde{\imath}$(Hom$(\mo_{\p^2},\mo_{\p^2}^{\oplus 3}))$, where $\tilde{\imath}:$ Hom$(\mo_{\p^2},E_0)\hookrightarrow$ Hom$(\mo_{\p^2}(-1),E_0)$ is the map induced by $\imath$.  Moreover $\forall j,j'\in H^0(\mo_{\p^2}(-1),E_0)-\{0\},j\not\equiv j'$,  Im$(j)\cap\textrm{Im}(j')=\emptyset$, and $\forall \imath,\imath'\in H^0(\mo_{\p^2}(-1),\mo_{\p^2})-\{0\},\imath\not\equiv \imath'$,  Im$(\imath)\cap\textrm{Im}(\imath')=\emptyset$. Hence all $\xi_f$ such that Im$(\xi_f)$ are contained in $\mmon$ but not $\mo_{\p^2}$ in $E_0$ are parametrized by $Gr(2,h^0(\mone))\times (\p(\textrm{Hom}(\mo_{\p^2}(-1),E_0))-\p(\textrm{Hom}(\mo_{\p^2},E_0))\times\p(\textrm{Hom}(\mo_{\p^2}(-1),\mo_{\p^2})))$ $\simeq \p^2\times(\p^7-\p^2\times\p^2)$.

We have the following lemma as a direct consequence.
\begin{lemma}\label{mcfivxo}$[\Xi_1]=[Gr(2,15)-Gr(2,6)\times\p^2-\p^2\times (\p^7-\p^2\times\p^2)].$
\end{lemma}
For a pair $(E,f)\in \Xi_2$, $f$ can be represented by the following matrix
\begin{equation}\label{romxt}\left(\begin{array}{ccccc}0&0&1&0&0\\0&0&0&1&0\\a_1&a_2&0&0&0\\a_3&a_4&0&0&b_1\\  a_5&a_6&0&0&b_2\end{array}\right),
\end{equation}
where $b_i\in H^0(\mone)$ and $a_i\in H^0(\mo_{\p^2}(2))$.  By Lemma \ref{scffi} $(E,f)$ is stable if and only if $kb_1\neq k'b_2,ka_1\neq k'a_2,\forall (k,k')\in\bc-\{0\}$. 

We write down the following two exact sequences.
\begin{equation}\label{xtcsmc}\xymatrix{ 0\ar[r]&\mo_{\p^2}(-1)\ar[r]^{~~(b_1,b_2)}&\mo_{\p^2}^{\oplus 2}\ar[r]^{f_r}& R_f\ar[r]&0}\end{equation}
\begin{equation}\label{scsxt}\xymatrix{ 0\ar[r]&\mo_{\p^2}(-2)\ar[r]^{~~(a_1,a_2)}&\mo_{\p^2}^{\oplus 2}\ar[r]^{f_s}& S_f\ar[r]&0}\end{equation}

$R_f\simeq I_1(1)$.  Either $S_f\simeq I_4(2)$ or $S_f$ lies in the following exact sequence. 
\begin{equation}\label{esxts}0\ra\mo_H(-1)\ra S_f\ra I_1(1)\ra0.
\end{equation} 
Isomorphism classes of $(R_f,S_f)$ are parametrized by $Hilb^{[1]}(\p^2)\times Gr(2,6)$.  

We have two commutative diagrams as follows.
\begin{equation}\label{xtmcbcd}\xymatrix@C=1.5cm@R=0.7cm{&0\ar[d] &0\ar[d] &0\ar[d]\\
0\ar[r]&\mo_{\p^2}(-3)\ar[r]^{~~(a_1,a_2)\otimes id_{\mmon}~~~}\ar[d]_{(b_1,b_2)\otimes id_{\mo_{\p^2}(-2)}}& \mo_{\p^2}(-1)^{\oplus 2}\ar[d]^{(b_1,b_2)^{\oplus 2}}\ar[r]^{~~~~f_s\otimes id_{\mmon}}&S_f(-1)\ar[r]\ar[d]^{id_{S_f}\otimes(b_1,b_2)}&0\\
0\ar[r]&\mo_{\p^2}(-2)^{\oplus 2}\ar[r]^{(a_1,a_2)^{\oplus 2}}\ar[d]_{f_r\otimes id_{\mo_{\p^2}(-2)}}&\mo_{\p^2}^{\oplus 4}\ar[d]^{f_r^{\oplus 2}}\ar[r]^{~~~f_s^{\oplus 2}}& S_f^{\oplus 2}\ar[r]\ar[d]^{id_{S_f}\otimes f_r}&0\\
0\ar[r]&R_f\otimes \mo_{\p^2}(-2)\ar[r]^{~~~id_{R_f}\otimes(a_1,a_2)}\ar[d] &R_f^{\oplus 2}\ar[d]\ar[r]^{id_{R_f}\otimes f_s}&R_f\otimes S_f\ar[r]\ar[d]&0\\
&0 &0 &0}
\end{equation} 
\begin{equation}\label{xtmcscd}\xymatrix@C=1.7cm{\mo_{\p^2}(-2)\ar[r]^{~~~(a_3,a_4)\oplus (a_5,a_6)}&\mo_{\p^2}^{\oplus 4}\ar[d]_{f_r^{\oplus 2}}\ar[r]^{~f_s^{\oplus 2}}& S_f^{\oplus 2}\ar[d]^{id_{S_f}\otimes f_r}\\
 &R_f^{\oplus 2}\ar[r]_{id_{R_f}\otimes f_s~~~} &R_f\otimes S_f.} 
\end{equation}

Isomorphism classes of $(E,f)\in \Xi_2$ are parametrized by $(R_f,S_f,\omega_f)$ with $\omega_f: \mo_{\p^2}(-2)\ra R_f\otimes S_f$ the composed map in (\ref{xtmcscd}).  Hence firstly we have a projective bundle over $Hilb^{[1]}(\p^2)\times Gr(2,6)$ with fibers isomorphic to $\p(H^0(R_f\otimes S_f(2)))\simeq \p^{15}$, which contains $\Xi_2$ as an open subset.  The complement of $\Xi_2$ in that projective bundle is the set of all $(R_f,S_f,\omega_f)$ such that Im$(\omega_f)$ are contained in the torsions of $R_f\otimes S_f$.

Torsion free $S_f$ are parametrized by $Gr(2,6)-\p^2\times\p^2$.  For $S_f$ torsion free, $R_f\otimes S_f$ has torsion if and only if $(a_1,a_2)|_x=0$ with $R_f\simeq I_x(1)$, and the nonzero torsion must be isomorphic to $\mo_x$.  Define $\mathcal{V}_1^i:=p_{*}(\mathcal{I}_1\otimes q^{*}\mo_{\p^2}(i))$ with $\mathcal{I}_1$, $p,q$ the same as before.  $\mathcal{V}^1_1$ and $\mathcal{V}^2_1$ are two vector bundles of rank 2 and 5 respectively over $Hilb^{[1]}(\p^2)$.  Hence $(R_f,S_f,\omega_f)$ with $S_f$ torsion free and Im$(\omega_f)$ contained in the torsion of $R_f\otimes S_f$ are parametrized by $Gr(2,\mathcal{V}^2_1)-Gr(2,\mathcal{V}^1_1)\times\p(H^0(\mone))\cup Gr(2,h^0(\mone))\times \p(\mathcal{V}_1^1)$, where $Gr(2,\mathcal{V}^i_1)$ is the relative Grassmannian of the vector bundle $\mathcal{V}^i_1$.  And
\begin{eqnarray}&&[Gr(2,\mathcal{V}^1_1)\times\p(H^0(\mone))\cup Gr(2,h^0(\mone))\times\p(\mathcal{V}_1^1)]\nonumber\\
&=&[Gr(2,\mathcal{V}^1_1)\times\p(H^0(\mone))]+[Gr(2,h^0(\mone))\times\p(\mathcal{V}_1^1)]-[\p(\mathcal{V}_1^1)] \nonumber\\&=&[Hilb^{[1]}(\p^2)\times(\p^2+\p^2\times\p^1-\p^1)].\nonumber\end{eqnarray}

Now let $S_f$ lie in (\ref{esxts}).  Write $R_f\simeq I_x(1)$ and $I_y(1)$ the quotient of $S_f$ in (\ref{esxts}).  If $x\neq y$, then $R_f\otimes I_y(1)$ is torsion free and in this case $(R_f,S_f,\omega_f)$ with Im$(\omega_f)$ contained in the torsions of $R_f\otimes S_f$ are parametrized by a projective bundle over $Hilb^{[1]}(\p^2)\times\p^2\times\p^2-Gr(2,\mathcal{V}^1_1)\times\p(H^0(\mone))$ with fibers isomorphic to $\p(H^0(I_{1}(1)\otimes\mo_H(1)))\simeq \p^2$.  

Finally we have a projective bundle over $Gr(2,\mathcal{V}^1_1)\times\p(H^0(\mone))$ with fibers isomorphic to $\p(H^0(\mo_x)\oplus H^0(I_1(1)\otimes\mo_H(1)))\simeq\p^3$ parametrizing $(R_f,S_f,\omega_f)$ such that $x=y$ and Im$(\omega_f)$ are contained in the torsions of $R_f\otimes S_f$.  Hence we have the following lemma.
\begin{lemma}\label{mfixtmt}$[\Xi_2]=[Hilb^{[1]}(\p^2)\times(Gr(2,6)\times\p^{15}-Gr(2,5)-\p^2\times\p^2\times\p^2-\p^2\times\p^3-\p^1+\p^2\times\p^2+\p^1\times\p^2+\p^2)].$
\end{lemma}

For a pair $(E,f)\in M_2^c$, $f$ can be represented by the following matrix
\[\left(\begin{array}{cccc}b_1&b_2&0&0\\ 0&0&1 &0\\ A_1&A_2&0 &B \end{array}\right),\]
where $b_i\in H^0(\mone)$, $A_i$ is a $3\times 1$ matrix with entries in $H^0(\mo_{\p^2}(2))$ and $B$ a $3\times 2$ matrix with entries in $H^0(\mone)$.  $(E,f)$ is stable, hence $kb_1\neq k'b_2, \forall (k,k')\in\bc^2-\{0\}$ and the parametrizing space $M_B$ of $B$ is of class $[Hilb^{[3]}(\p^2)-|H|\times\p^3+|H|]$ by Lemma \ref{doof} and Lemma \ref{icqff}.
We write down the following two exact sequences.
\begin{equation}\label{focsmc}\xymatrix{ 0\ar[r]&\mo_{\p^2}(-1)^{\oplus 2}\ar[r]^{~~~~f_{B^t}}&\mo_{\p^2}^{\oplus 3}\ar[r]^{f_r}& R_f\ar[r]&0}\end{equation}
\begin{equation}\label{scsfo}\xymatrix{ 0\ar[r]&\mo_{\p^2}(-1)\ar[r]^{~~(b_1,b_2)}&\mo_{\p^2}^{\oplus 2}\ar[r]^{f_s}& S_f\ar[r]&0}\end{equation}

$S_f\simeq I_1(1).$  Either $R_f\simeq I_3(2)$ or $R_f$ lies in the following exact sequence. 
\begin{equation}\label{esfos}0\ra\mo_H(-1)\ra R_f\ra \mone\ra0.
\end{equation} 
Isomorphism classes of $(R_f,S_f)$ are parametrized by $M_B\times Hilb^{[1]}(\p^2)$.  

We have two commutative diagrams as follows.
\begin{equation}\label{fomcbcd}\xymatrix@C=1.5cm@R=0.7cm{&0\ar[d] &0\ar[d] &0\ar[d]\\
0\ar[r]&\mo_{\p^2}(-2)^{\oplus 2}\ar[r]^{~~((b_1,b_2)\otimes id_{\mmon})^{\oplus 2}~~~}\ar[d]_{f_{B^t}\otimes id_{\mo_{\p^2}(-1)}}& \mo_{\p^2}(-1)^{\oplus 4}\ar[d]^{f_{B^t}^{\oplus 2}}\ar[r]^{~~~~(f_s\otimes id_{\mmon})^{\oplus 2}}&S_f(-1)^{\oplus 2}\ar[r]\ar[d]^{id_{S_f}\otimes(b_1,b_2)}&0\\
0\ar[r]&\mo_{\p^2}(-1)^{\oplus 3}\ar[r]^{(b_1,b_2)^{\oplus 3}}\ar[d]_{f_r\otimes id_{\mo_{\p^2}(-1)}}&\mo_{\p^2}^{\oplus 6}\ar[d]^{f_r^{\oplus 2}}\ar[r]^{~~~f_s^{\oplus 3}}& S_f^{\oplus 3}\ar[r]\ar[d]^{id_{S_f}\otimes f_r}&0\\
0\ar[r]&R_f\otimes \mo_{\p^2}(-1)\ar[r]^{~~~id_{R_f}\otimes(b_1,b_2)}\ar[d] &R_f^{\oplus 2}\ar[d]\ar[r]^{id_{R_f}\otimes f_s}&R_f\otimes S_f\ar[r]\ar[d]&0\\
&0 &0 &0}
\end{equation} 
\begin{equation}\label{fomcscd}\xymatrix@C=1.7cm{\mo_{\p^2}(-2)\ar[r]^{A_1^t\oplus A_2^t}&\mo_{\p^2}^{\oplus 6}\ar[d]_{f_r^{\oplus 2}}\ar[r]^{~f_s^{\oplus 2}}& S_f^{\oplus 3}\ar[d]^{id_{S_f}\otimes f_r}\\
 &R_f^{\oplus 2}\ar[r]_{id_{R_f}\otimes f_s~~~} &R_f\otimes S_f.} 
\end{equation}

Isomorphism classes of $(E,f)\in M_2^c$ are parametrized by $(R_f,S_f,\omega_f)$ with $\omega_f: \mo_{\p^2}(-2)\ra R_f\otimes S_f$ the composed map in (\ref{fomcscd}).  Hence we have a projective bundle over $M_B\times Hilb^{[1]}(\p^2)$ with fibers isomorphic to $\p(H^0(R_f\otimes S_f(2)))\simeq \p^{16}$, which contains $M_2^c$ as an open subset.

We need to exclude all the points $(R_f,S_f,\omega_f)$ that Im$(\omega_f)$ are contained in the torsions of $R_f\otimes S_f$.  Firstly let $R_f$ lie in (\ref{esfos}), then the torsion of $R_f\otimes S_f$ is isomorphic to $\mo_H(-1)\otimes I_1(1)$.  These $(R_f,S_f,\omega_f)$ are parametrized by a projective bundle over $(\displaystyle{\cup_{i=0}^2}\mathbb{A}^i)\times Hilb^{[1]}(\p^2)$ with fibers isomorphic to $\p(H^0(\mo_{H}(1)\otimes I_1(1)))\simeq\p^2.$

Let $R_f\simeq I_3(2)$.  Denote $S_f\simeq I_x(1)$.  The torsion of $R_f\otimes S_f$ is a linear subspace of $\mo_x^{\oplus 2}\simeq\bc^2$ which is the kernel of $B^t|_x$.  If $x\not\in Supp(\mo_{\p^2}(2)/R_f)$, $R_f\otimes S_f$ is torsion free.  If $Supp(\mo_{\p^2}(2)/R_f)=\{x,y,z\}$ with $y,z\neq x$, for simplicity we let $x=[0,0,1],$ $y=[0,1,0]$ and $z=[1,0,0]$ and the matrix $B$ have the following form.
\[\left(\begin{array}{cc}x_1&0\\ x_0&x_0\\ 0&x_2\end{array}\right).\]
Hence for this case $Tor(R_f\otimes S_f) \simeq \mo_x$.  

Let $R_f\simeq I_{\{x,2y\}}(2)$, then $B$ can be 
\[\left(\begin{array}{cc}x_0&0\\ x_2&x_0\\ 0&x_1\end{array}\right).\]
Hence for this case $Tor(R_f\otimes S_f) \simeq \mo_x$.  

Let $R_f\simeq I_{\{2x,y\}}(2)$, then $B$ can be
\[\left(\begin{array}{cc}x_0&0\\ x_1&x_0\\ 0&x_2\end{array}\right).\]
Hence for this case $Tor(R_f\otimes S_f) \simeq \mo_x$.  

Let $R_f\simeq I_{\{3x\}}(2)$, then $B$ can be 
\[\left(\begin{array}{cc}x_0&0\\ x_1&x_0\\ kx_2&x_1\end{array}\right),for~any~k\in\bc.\]
Hence for this case $Tor(R_f\otimes S_f) \simeq \mo_x$ if $k\neq 0$, $Tor(R_f\otimes S_f) \simeq \mo_x^{\oplus 2}$ if $k=0$.  

The projective bundle $\p(\mathcal{V}_1^1)$ as defined before over $Hilb^{[1]}(\p^2)$ parametrizes all $(x,C)$ with $x$ a single point and $C$ a curve of degree 1 passing through $x$.  Hence we have the universal family $\overline{\mc}_1\subset \p^2\times \p(\mathcal{V}_1^1).$  Denote $\mathcal{Z}_1$ to be the universal family of subschemes in $Hilb^{[1]}(\p^2)\times\p^2$ and $\pi:\p(\mathcal{V}_1^1)\ra Hilb^{[1]}(\p^2)$ the projection.  Define $\overline{\mc}_1^0:=\overline{\mc}_1-(\pi\times id_{\p^2})^{*}\mathcal{Z}_1$.  Denote $\p(\mathcal{V}_1^1)^{[2]}$ the relative Hilbert scheme of 2-points on $\p(\mathcal{V}_1^1)$ over $Hilb^{[1]}(\p^2)$.  There is a natural embedding $\imath:\p(\mathcal{V}_1^1)\hookrightarrow \p(\mathcal{V}_1^1)^{[2]}$ sending every point to the double-point supported at it.  We have the following diagram
\begin{equation}\label{mcfimx}\xymatrix{\overline{\mathcal{C}}_1^0\times_{\p(\mathcal{V}_1^1)\times_{\pi}\p(\mathcal{V}_1^1)}\overline{\mathcal{C}}_1^0- p^{*}\Delta(\p(\mathcal{V}_1^1))\ar[r]^{~~~~~~~~~~~~~~\delta'}\ar[d]_{p} &\mathcal{X}\ar[d]^{p'}\\ 
\p(\mathcal{V}_1^1)\times_{\pi}\p(\mathcal{V}_1^1) -\Delta(\p(\mathcal{V}_1^1))\ar[r]_{~~~~~~\delta}&\p(\mathcal{V}_1^1)^{[2]}-\imath(\p(\mathcal{V}_1^1)),}\end{equation}
with $\Delta$ the diagonal embedding and $\mathcal{X}$ defined to make (\ref{mcfimx}) a Cartesian diagram.  Notice that a priori $\mathcal{X}$ may not exist, but if it exists, it parametrizes isomorphism classes of $(R_f,S_f,\omega_f)$ with $S_f\simeq I_x(1)$, $R_f\simeq I_{\{x,y,z\}}(2)$ for $\{x,y,z\}\in N_4^0$ i.e. $H^0(I_{\{x,y,z\}}(1))=0$, and Im$(\omega)\subset Tor(R_f\otimes S_f)$.

\begin{lemma}\label{fpins}$\mathcal{X}$ exists and $[\mathcal{X}]=[Hilb^{[1]}(\p^2)\times(\p^2-\p^1)\times(\p^1-1)\times(\p^1-1)].$
\end{lemma}
\begin{proof}Take an affine cover of $Hilb^{[1]}(\p^2)=\cup_{i}U_i$ with $U_i\simeq \mathbb{A}^2$.  It is enough to prove the lemma with $Hilb^{[1]}(\p^2)$ replaced by $U_i$.  Denote by $\mathcal{Z}_1|_{U_i}$,  $\p(\mathcal{V}_1^1)|_{U_i}$, $\p(\mathcal{V}_1^1)^{[2]}|_{U_i}$, $\overline{\mc}_1|_{U_i}$ and $\overline{\mc}_1^0|_{U_i}$ the pull back of these schemes via the open embedding $U_i\hookrightarrow Hilb^{[1]}(\p^2)$.  Then we have that $\mathcal{Z}_1|_{U_i}\simeq U_i$, $\p(\mathcal{V}^1_1)|_{U_i}\simeq U_i\times\p^1$, $\p(\mathcal{V}_1^1)^{[2]}|_{U_i}\simeq U_i\times \p^2$, $\overline{\mc}_1|_{U_i}\simeq U_i\times\p^1\times\p^1$ and $\overline{\mc}_1^0|_{U_i}\simeq U_i\times\p^1\times\mathbb{A}^1$.  Hence (\ref{mcfimx}) becomes the following commutative diagram.
\begin{equation}\label{nwfimx}\xymatrix{U_i\times(\p^1\times\p^1-\Delta(\p^1))\times \mathbb{A}^1\times \mathbb{A}^1\ar[r]^{\qquad\qquad\qquad\qquad~~\delta_i'}\ar[d]_{p_i} &\mathcal{X}_i\ar[d]^{p'_i}\\ 
U_i\times(\p^1\times\p^1 -\Delta(\p^1))\ar[r]_{~~\delta_i}&U_i\times (\p^2-\imath(\p^1)),}\end{equation}
with $\mathcal{X}_i\simeq U_i\times(\p^2-\imath(\p^1))\times\mathbb{A}^1\times\mathbb{A}^1$.  Hence the lemma.
\end{proof}
Isomorphism classes of $(R_f,S_f)$ with $S_f\simeq I_x(1)$, $R_f\simeq I_{\{2x,y\}}(2)$ for $H^0(I_{\{2x,y\}}(1))=0$ and Im$(\omega)\subset Tor(R_f\otimes S_f)$ are parametrized by $(\p(\mathcal{V}_1^1)\times_{Hilb^{[1]}(\p^2)}\overline{\mc}_1^0)-\Delta$, where $\Delta$ is defined by the following Cartesian diagram
\[\xymatrix{\Delta\ar[r]\ar[d]&\overline{\mc}_1^0\ar[d]\\ \p(\mathcal{V}_1^1)\ar[r]^{id}&\p(\mathcal{V}_1^1)\ar[r]^{\pi~~~~}& Hilb^{[1]}(\p^2)}\]  
\begin{lemma}\label{fpstwo}$[(\p(\mathcal{V}_1^1)\times_{Hilb^{[1]}(\p^2)}\overline{\mc}_1^0)-\Delta]=[Hilb^{[1]}(\p^2)\times\p^1\times(\p^1-1)\times(\p^1-1)].$
\end{lemma}
\begin{proof}Take the affine cover $Hilb^{[1]}(\p^2)=\cup_{i}U_i$ with $U_i\simeq \mathbb{A}^2$.  Replace $Hilb^{[1]}(\p^2)$ by $U_i$ and the lemma follows immediately.
\end{proof}

The normal sheaf of $\overline{\mc}_1$ in $\p^2\times\p(\mathcal{V}^1_1)$ is locally free over $\overline{\mc}_1^{0}$.  We denote by $\mathcal{N}_{\mc}^0$ the total space of the normal bundle over $\overline{\mc}_1^0$.  Then isomorphism classes of $(R_f,S_f,\omega_f)$ with $S_f\simeq I_x(1)$, $R_f\simeq I_{\{x,2y\}}(2)$ for $H^0(I_{\{x,2y\}}(1))=0$ and Im$(\omega)\subset Tor(R_f\otimes S_f)$ are parametrized by $\mathcal{N}_{\mc}^0$.
\begin{lemma}\label{fpsth}$[\mathcal{N}_{\mc}^0]=[Hilb^{[1]}(\p^2)\times\p^1\times (\p^1-1)\times(\p^1-1)].$
\end{lemma}
\begin{proof}$[\mathcal{N}_{\mc}^0]=[\mathbb{A}^1\times(\overline{\mc}_1-(\pi\times id_{\p^2})^{*}\mathcal{Z}_1)]$.  Also we see $[(\pi\times id_{\p^2})^{*}\mathcal{Z}_1]=[\p^1\times Hilb^{[1]}(\p^2)]$ and $[\overline{\mc}_1]=[Hilb^{[1]}(\p^2)\times\p^1\times\p^1]$.  Hence the lemma.
\end{proof}

If $R_f\simeq I_{\{3x\}}(2)$ with $H^0(R_f(-1))=0$ and $Tor(R_f\otimes I_x) \simeq \mo_x$, then $R_f$, viewed as an ideal of $\widehat{\mo}_{\p^2,x}\simeq\bc[[x_0,x_1]]$, is generated by $(kx_0-x_1^2,\mathfrak{m}^3)$ with $\mathfrak{m}$ the maximal ideal in $\bc[[x_0,x_1]]$ and $k\in\bc^{*}$. Hence such $R_f$ are parametrized by $(x_0,k)$ for any fixed $x\in\p^2$.  Hence isomorphism classes of these $(R_f,S_f,\omega_f)$ such that Im$(\omega)\subset Tor(R_f\otimes S_f)$ are parametrized by $\p(\mathcal{V}^1_1)\times (\mathbb{A}^1-\{0\})$.

Finally, let $S_f\simeq I_x(1)$, $R_f\simeq I_{\{3x\}}(2)$ with $H^0(R_f(-1))=0$ and $Tor(R_f\otimes S_f)\simeq\mo_x^{\oplus 2}$, then $R_f$ is determined by $x$ since $R_f(-2)\simeq I_{x}^2$.  Hence isomorphism classes of these $(R_f,S_f,\omega_f)$ such that Im$(\omega)\subset Tor(R_f\otimes S_f)$ are parametrized by $Hilb^{[1]}(\p^2)\times \p(H^0(\mo_x^{\oplus 2}))\simeq Hilb^{[1]}(\p^2)\times \p^1.$

\begin{lemma}\label{fimcmt}$[M^c_2]=[M_B\times Hilb^{[1]}(\p^2)\times\p^{16}-\p^2\times\p^2\times\p^2-\p^2\times\p^1\times\p^1\times(\p^1-1)-\p^2\times\p^2\times(\p^2-\p^1)],$ with $[M_B]=[Hilb^{[3]}(\p^2)-\p^2\times\p^3+|H|]$. 
\end{lemma}
\begin{proof}$[M_2^c]=[M_B\times Hilb^{[1]}(\p^2)\times\p^{16}-\p^2\times |H|\times Hilb^{[1]}(\p^2)-\p(\mathcal{V}^1_1)\times (\mathbb{A}^1-1)- Hilb^{[1]}(\p^2)\times \p^1-\mathcal{X}-\p(\mathcal{V}^1_1)\times_{Hilb^{[1]}(\p^2)}\overline{\mc}_1^0+\Delta-\mathcal{N}_{\mc}^0]$.

By Lemma \ref{fpins}, Lemma \ref{fpstwo} and Lemma \ref{fpsth}, we get the lemma by direct computation.
\end{proof}
\begin{proof}[Proof of Theorem \ref{mtfi} for $M(5,2)$] We have 
\[[M(5,2)]=[M'_3]+[M_3]+[\Xi_1]+[\Xi_2]+[M_2^c].\] 
Combine Lemma \ref{mcfimthp}, Lemma \ref{mcfivmth}, Lemma \ref{mcfivxo}, Lemma \ref{mfixtmt} and Lemma \ref{fimcmt}, we get the result by direct computation.
\end{proof}

\section*{Appendix}
\appendix
\section{Proofs of Lemma \ref{scffi} and Lemma \ref{scffiv}.}
\begin{lemma}[\textbf{Lemma \ref{scffi}}]\label{ALe}A pair $(E,f)$ with $rank(E)=5$ and $deg(E)=-4$ is stable if and only if for any two direct summands $D',D''$ of $E$ such that $D'\simeq D''$ and $f(D'\otimes\mmon)\subset D''$,  we have $\mu(D')<\mu(E)$.\end{lemma}
\begin{proof}We first prove the lemma for $E\simeq\mo_{\p^2}^{\oplus 2}\oplus\mmon^{\oplus 2}\oplus\mo_{\p^2}(-2)$.  We want to show that if $\exists E'\subset E$ a direct sum of line bundles with $\mu(E')>\mu(E)$ and $f^{-1}(E')\simeq E'\otimes\mmon$, then $\exists D,D'\subset E$ two direct summands with $D\simeq D'$ and $\mu(D)>\mu(E)$, such that $f(D\otimes\mmon)\subset D'$.  With no loss of generality, we assume that $E'$ has the form $\bigoplus_{i}\mo_{\p^2}(n_i)^{\oplus a_i}$ with $a_i>0$ and $n_{i}-n_{i+1}=1$.  

Let $E'\simeq E''\subset E$ with $E''$ not a direct summand of $E$.  Then $E''$ has to be one of the following three cases:

(1) $E''\subset\mo_{\p^2}^{\oplus 2}$ and $E''\simeq \mo_{\p^2}\oplus\mmon$;

(2) $E''\subset\mo_{\p^2}^{\oplus 2}\oplus\mmon$ and $E''\simeq \mo_{\p^2}\oplus\mmon^{\oplus 2};$

(3) $E''\subset\mo_{\p^2}^{\oplus 2}\oplus\mmon^{\oplus 2}$ and $E''\simeq \mo_{\p^2}^{\oplus 2}\oplus\mmon\oplus\mo_{\p^2}(-2).$

By Nakayama's lemma, we know that $E''\otimes\mmon$ can't be the preimage of any direct summand of $E$ and also $f^{-1}(E'')= E''\otimes\mmon\Rightarrow f(D\otimes\mmon)\subset D$ with $D$ the smallest direct summand of $E$ containing $E''$.

So we assume that $f^{-1}(E'')=E'\otimes\mmon$ with $E'$ a direct summand of $E$ isomorphic to $E''$.

Let $E''$ be in case (1).  By the assumption we have $f(E'\otimes \mmon)\subset \mo_{\p^2}^{\oplus 2}$.  On the other hand, write $E=\mo_{\p^2}\oplus E'\oplus\mmon\oplus \mo_{\p^2}(-2)$, so for the other direct summand $\mo_{\p^2}$ we have $f(\mo_{\p^2}\otimes\mmon)\subset\mo_{\p^2}^{\oplus 2}\oplus\mmon$.  Hence $f((\mo_{\p^2}\oplus E')\otimes \mmon)\subset \mo_{\p^2}^{\oplus 2}\oplus\mmon$, and hence we get $D=D'=E'\oplus\mo_{\p^2}=\mo_{\p^2}^{\oplus 2}\oplus \mmon$.
 
Case (2) is analogous to case (1).

Let $E''$ be in case (3).  By the assumption we have $f(E'\otimes\mmon)\subset\mo_{\p^2}^{\oplus 2}\oplus\mmon^{\oplus 2}$.  Write $E=E'\oplus L$ with $L\simeq\mmon$.  We can ask $f$ to identify $L\otimes \mmon$ with the summand $\mo_{\p^2}(-2)$ in $E$.  

Denote by $f_o:E'\otimes \mmon\ra E''$ the restriction of $f$.  If $f_o((\mo_{\p^2}^{\oplus 2}\oplus\mmon)\otimes\mmon)\subset\mo_{\p^2}^{\oplus 2}\oplus\mmon$, then we have $D=D'=\mo_{\p^2}^{\oplus 2}\oplus\mmon\subset E'$.  

If $f_o((\mo_{\p^2}^{\oplus 2}\oplus\mmon)\otimes\mmon)\not\subset\mo_{\p^2}^{\oplus 2}\oplus\mmon$, $f_o$ induces an isomorphism from the direct summand $\mmon\otimes \mmon$ of $E'\otimes\mmon$ to the direct summand $\mo_{\p^2}(-2)$ of $E''$.  Hence we can ask $f_o$ to identify these two direct summands.  Write $E'=\mmon\oplus L'$ with $L'\simeq \mo_{\p^2}^{\oplus 2}\oplus\mo_{\p^2}(-2)$, then we have $f_o(L'\otimes\mmon)\subset \mo_{\p^2}^{\oplus 2}\oplus \mmon$.  Moreover because $f$ identifies $L\otimes\mmon$ with the summand $\mo_{\p^2}(-2)$ in $E$, $f((L\oplus L')\otimes\mmon)$ is contained in the direct summand $\mo_{\p^2}^{\oplus 2}\oplus \mmon\oplus\mo_{\p^2}(-2)$ of $E$, hence we have $D=L\oplus L'$ and $D'$ is the direct summand $\mo_{\p^2}^{\oplus 2}\oplus \mmon\oplus\mo_{\p^2}(-2)$ containing $f(D\otimes\mmon)$.

This finishes the proof for $E\simeq\mo_{\p^2}^{\oplus 2}\oplus\mmon^{\oplus 2}\oplus\mo_{\p^2}(-2)$.

Let $E\simeq\mone\oplus\mo_{\p^2}\oplus\mmon\oplus\mo_{\p^2}(-2)^{\oplus 2}$.  We have the following six possibilities for $E''$. 

(4) $E''\subset\mo_{\p^2}(1)$ and $E''\simeq \mo_{\p^2}$;

(5) $E''\subset\mone\oplus\mo_{\p^2}$ and $E''\simeq \mo_{\p^2}^{\oplus 2};$

(6) $E''\subset\mone\oplus\mo_{\p^2}$ and $E''\simeq \mo_{\p^2}\oplus\mmon;$

(7) $E''\subset\mone\oplus\mo_{\p^2}(-1)$ and $E''\simeq \mo_{\p^2}\oplus\mmon;$

(8) $E''\subset\mone\oplus\mo_{\p^2}\oplus\mmon$ and $E''\simeq \mo_{\p^2}^{\oplus 2}\oplus\mmon;$

(9) $E''\subset\mone\oplus\mo_{\p^2}\oplus\mmon$ and $E''\simeq \mo_{\p^2}\oplus\mmon^{\oplus 2}.$

Analogously $E''\otimes\mmon$ can not be the preimage of any direct summand of $E$ and also $f^{-1}(E'')=E''\otimes\mmon\Rightarrow f(D\otimes\mmon)\subset D$ with $D$ the smallest direct summand of $E$ containing $E''$.  Let $E_3''$ and $E_4''$ be the bundles in case (6) and (7) respectively.  $f^{-1}(E''_3)=E_4''\otimes\mmon\Rightarrow f((\mone\oplus\mo_{\p^2}\oplus\mmon)\otimes\mmon)\subset\mone\oplus\mo_{\p^2}\oplus\mmon$,  and $f^{-1}(E''_4)=E_3''\otimes\mmon\Rightarrow f(\mone\otimes\mmon)\subset\mone.$  

Hence we then assume $E'$ a direct summand of $E$ isomorphic to $E''$ and $f^{-1}(E'')=E'\otimes\mmon$.

For case (4), by assumption we have $f(\mo_{\p^2}\otimes\mmon)\subset\mone\oplus\mo_{\p^2}$ hence $D=D'= \mone\oplus\mo_{\p^2}$.

Bundles in case (5), case (8) and case (9) can not be direct summands of $E$, hence these three cases are done.

For case (6), by assumption we have $f((\mo_{\p^2}\oplus\mmon)\otimes\mmon)\subset\mone\oplus\mo_{\p^2}\oplus\mmon$ hence $D=D'=\mone\oplus\mo_{\p^2}\oplus\mmon$.

For case (7), by assumption we have $f(\mo_{\p^2}\otimes\mmon)\subset\mone\oplus\mo_{\p^2}(-1)$ hence $D=D'=\mone\oplus\mo_{\p^2}\oplus\mmon$.

This finishes the proof for the whole lemma.
\end{proof}

\begin{lemma}[\textbf{Lemma \ref{scffiv}}]A pair $(E,f)$ with $rank(E)=5$ and $deg(E)=-3$ is stable if and only if for any two direct summands $D',D''$ of $E$ such that $D'\simeq D''$ and $f(D'\otimes\mmon)\subset D''$,  we have $\mu(D')<\mu(E)$.\end{lemma}
\begin{proof}We use the same notations as in the proof of Lemma \ref{ALe}, we list out all the possibilities of $E''$ as follows.

Let $E\simeq \mo_{\p^2}^{\oplus 2}\oplus\mmon^{\oplus 3}.$

(1) $E''\subset\mo_{\p^2}^{\oplus 2}$ and $E''\simeq \mo_{\p^2}\oplus\mmon$;

Let $E\simeq \mo_{\p^2}^{\oplus 3}\oplus\mmon\oplus \mo_{\p^2}(-2).$

(2) $E''\subset\mo_{\p^2}^{\oplus 3}$ and $E''\simeq \mo_{\p^2}\oplus\mmon$;

(3) $E''\subset\mo_{\p^2}^{\oplus 3}$ and $E''\simeq \mo_{\p^2}^{\oplus 2}\oplus\mmon$;

(4) $E''\subset\mo_{\p^2}^{\oplus 3}\oplus\mmon$ and $E''\simeq \mo_{\p^2}^{\oplus 2}\oplus\mmon^{\oplus 2}$;

Let $E\simeq \mone\oplus\mo_{\p^2}\oplus\mmon^{\oplus 2}\oplus\mo_{\p^2}(-2).$

(5) $E''\subset\mo_{\p^2}(1)$ and $E''\simeq \mo_{\p^2}$;

(6) $E''\subset\mone\oplus\mo_{\p^2}$ and $E''\simeq \mo_{\p^2}^{\oplus 2};$

(7) $E''\subset\mone\oplus\mo_{\p^2}$ and $E''\simeq \mo_{\p^2}\oplus\mmon;$

(8) $E''\subset\mone\oplus\mo_{\p^2}(-1)$ and $E''\simeq \mo_{\p^2}\oplus\mmon;$

(9) $E''\subset\mone\oplus\mo_{\p^2}\oplus\mmon$ and $E''\simeq \mo_{\p^2}^{\oplus 2}\oplus\mmon;$

(10) $E''\subset\mone\oplus\mo_{\p^2}\oplus\mmon^{\oplus 2}$ and $E''\simeq \mone\mo_{\p^2}\oplus\mmon\oplus\mo_{\p^2}(-2).$

Cases (1) (5) (6) (7) (8) (9) are the same as cases (1) (4) (5) (6) (7) (8) in Lemma \ref{ALe} respectively.  Case (10) is analogous to case (3) in Lemma \ref{ALe}.  Cases (2) (3) (4) are analogous to case (1).  Hence the lemma.  
\end{proof}

\section{Proof of Lemma \ref{ctome}.}
\begin{lemma}[\textbf{Lemma \ref{ctome}}]$[\Omega_{2}^{[6]}]=\bl^{11}+3\bl^{10}+8\bl^{9}+18\bl^{8}+30\bl^{7}+39\bl^{6}+38\bl^{5}+28\bl^{4}+15\bl^{3}+6\bl^{2}+2\bl^{1}+1$.
\end{lemma}
\begin{proof}Denote by $\mc_2$ the universal curve in $\p^2\times |2H|$ and $\mc^{[6]}_2$ the relative Hilbert scheme of 6-points on $\mc_2$ over $|2H|$.  We have a surjective map $\xi:\mc_2^{[6]}\ra\Omega_2^{[6]}$.  The fiber of $\xi$ over $I_{\{x_1,\ldots,x_6\}}$ consists of all curves passing through $x_1,\ldots,x_6$ and hence isomorphic to $\p(H^0(I_{\{x_1,\ldots,x_6\}}(2)))$.  Let $\ms_n:=\{I_6\in\Omega_2^{[6]}|h^0(I_6(2))=n+1\}$, then $\Omega_2^{[6]}=\coprod_{n=0}^2 \ms_n.$  Define $\mr_n:=\xi^{-1}(\ms_n)$, then $\mr_n\simeq \p(p_{*}(\mathcal{I}(2)|_{\p^2\times\ms_n}))$ is a projective bundle over $\ms_n$ with fibers isomorphic to $\p^n$.

Denote by $\mc_2^o$ the family of integral curves in $|2H|$.  $\mc^o_2$ is open in $\mc_2$.  \begin{lemma}\label{ABic}$[{\mc_2^o}^{[6]}]=[(|2H|-Sym^2(|H|))\times \p^6]$ with $Sym^2(|H|)$ the symmetric power of order 2 of $|H|$.
\end{lemma}
\begin{proof}The subspace $|2H|^o$ in $|2H|$ parametrizing integral curves is $|2H|-Sym^2(|H|)$.  ${\mc_2^o}^{[6]}$ is a projective bundle over $|2H|^o$ with fibers isomorphic to $\p^6$.  Hence the lemma.
\end{proof}

Denote by $\mc^R_2\twoheadrightarrow (Sym^2(|H|)-|H|)$ and $\mc^N_2\twoheadrightarrow |H|$ the families of reducible curves and non-reduced curves in $|2H|$ respectively.  Let $C^R$ be a reducible curve in $|2H|$ and $C^N$ a non-reduced curve.  Denote $R_n^R~(\mr_n^R)=\mr_n\cap Hilb^{[6]}(C^R)~({\mc^R_2}^{[6]})$ and $R_n^N~(\mr_n^N)=\mr_n\cap Hilb^{[6]}(C^N)~({\mc^N_2}^{[6]})$.  Then we have the following lemma
\begin{lemma}\label{ABg}$[\mr^N_n]=[R_n^N\times |H|],$ for $n=0,1,2$.
\end{lemma}
\begin{proof}We can take an affine cover of $|H|$, write $|H|=\cup_j V_j$ such that $\mc^N_2|_{V_j}\simeq V_j\times C^N$.  Hence the lemma.
\end{proof}
Denote $\mathfrak{m}=<x,y>$ the maximal ideal of $\mathtt{S}:=\bc[[x,y]]$.  To study $R^N_n$ for $n=0,1,2$, we write down a table for ideals in $\ts/(x^2)$ as Table I.

\begin{table}[h]
\centering
\begin{tabular}{|c|c|c|}
\hline\multicolumn{3}{|l|}{\textbf{Table I}~~\qquad\qquad\qquad\qquad Ideals $I$ of $\ts$ containing $(x^2)$}\\
\hline Co-length of $I$ & Ideal $I$ & $I\cap (\mathfrak{m}^2-\mathfrak{m}^3)$\\
\hline 1 & $\mathfrak{m}$ & $\bc x^2\oplus \bc xy\oplus\bc y^2$\\ \hline
2 & $\mathfrak{m}^2+(kx+k'y)\ts, (k,k')\neq0$ &$\bc x^2\oplus \bc xy\oplus\bc y^2$\\ \hline
\multirow{2}{*}{3} & $\km^2$ &  $\bc x^2\oplus \bc xy\oplus\bc y^2$\\ \cline{2-3}
&$\km^3+(x+ky^2)\mathtt{S}$ & $\bc x^2\oplus \bc xy$ \\ \hline
\multirow{3}{*}{4}& $x^2\ts+(ky^2+k'xy)\ts+\km^3,(k,k')\neq0$ & $\bc x^2\oplus \bc(ky^2+k'xy)$\\ \cline{2-3}
&$(x+ky^2+k'y^3)\ts+\km^4,k\neq0$& $\bc x^2$\\ 
\cline{2-3}&$(x+k'y^3)\ts+\km^4$& $\bc x^2\oplus \bc xy$\\ 
\hline
\multirow{5}{*}{5} &$x^2\ts+\km^3$&$ \bc x^2$\\ \cline{2-3}
&$x^2\ts+(xy+ky^2+k'y^3)\ts+\km^4, k\neq 0$&$\bc x^2$\\ \cline{2-3}
&$x^2\ts+(xy+k'y^3)\ts+\km^4$&$\bc x^2\oplus\bc xy$\\ \cline{2-3}
&$(x+ky^3+k'y^4)\ts+\km^5,k\neq0$&$\bc x^2$\\ \cline{2-3}
&$(x+k'y^4)\ts+\km^5$&$\bc x^2\oplus\bc xy$\\ \hline
\multirow{5}{*}{6}&$x^2\ts+(kxy^2+k'y^3)\ts+\km^4,(k,k')\neq0$&$ \bc x^2$\\ \cline{2-3}
&$x^2\ts+(xy+ky^3+k'y^4)\ts+\km^5, k\neq 0$&$\bc x^2$\\ \cline{2-3}
&$x^2\ts+(xy+k'y^4)\ts+\km^5$&$\bc x^2\oplus\bc xy$\\ \cline{2-3}
&$(x+ky^3+k'y^4+k''y^5)\ts+\km^6,(k,k'')\neq0$&$\bc x^2$\\ \cline{2-3}
&$(x+k''y^5)\ts+\km^6$&$\bc x^2\oplus\bc xy$\\ \hline

\end{tabular}
\end{table}
Let $C^N_r$ be the reduced curve supported on $C^N$, then $C^N_r\simeq \p^1$.  Denote by $S^i$ the subset in $Hilb^{[6]}(C^N)$ consisting of $[I_i^r\cap I_{6-i}]$ with $[I_i^r]\in Hilb^{[i]}(C^N_r)$ and $i$ maximal for this expression.  $S^i_n:=S^i\cap R_n^N$.  We then have 
\begin{lemma}\label{xsid}$S_n^i$ are empty except the following 9 terms:

$[S_2^6]=[\p^6];$  $[S^4_1]=[\p^2\times\p^1\times \mathbb{A}^1];$ $[S^3_1]=[\p^2\times\p^1\times\mathbb{A}^4];$

$[S^2_0]=[\p^2\times\p^1\times \mathbb{A}^5],$ $[S^2_2]=[\p^2\times\p^2\times\mathbb{A}^3];$

$[S^1_0]=[\p^2\times\p^1\times \mathbb{A}^4],$ $[S^1_2]=[\p^2\times\p^2\times\mathbb{A}^3];$

$[S^0_0]=[\p^2\times\p^1\times (\mathbb{A}^4+2\mathbb{A}^3+\mathbb{A}^2)+\p^2\times\mathbb{A}^4+\p^2(\p^3-\p^1\times\p^1)\times\mathbb{A}^3],$ 

$[S^0_1]=[\p^2\times\p^1\times\mathbb{A}^1].$

\end{lemma}
We omit the proof of Lemma \ref{xsid} since it can be done by elementary analysis and computation case by case.  Lemma \ref{xsid} together with Lemma \ref{ABq} gives $[\mr^N_n]$ for $n=0,1,2.$

To compute $[\mr^R_n]$, we first define $\widetilde{\mc_2^R}$ by the following Cartesian diagram.
\[\xymatrix{\widetilde{\mc_2^R}\ar[r]^{\pi_1}\ar[d]&\mc_2^R\ar[d]\\ \p^2\times\p^2-\Delta(\p^2)\ar[r]^{\pi} & Sym^2(\p^2)-\p^2,}\]
where $\pi$ is the quotient of the free action of the order two permutation group $\sigma_2$.  The action of $\sigma_2$ lifts to $\widetilde{\mc_2^R}^{[6]}$ with ${\mc_2^R}^{[6]}$ the quotient.  Recall that $R^R_n~(\mr^R_n)=\mr_n\cap Hilb^{[6]}(C^R)~({\mc_2^R}^{[2]})$.  Let $\widetilde{\mr^R_n}:=\pi_2^{-1}(\mr_n^R)$ with $\pi_2$ the lift of $\pi.$  $\mr_n^R$ is the quotient of $\widetilde{\mr_n^R}$ by the action of $\sigma_2$.
\begin{lemma}\label{ABr}$[\widetilde{\mr^R_n}]=[(\p^2\times\p^2-\Delta(\p^2))\times R_n^R]$ for $n=0,1,2.$
\end{lemma}
\begin{proof}Analogous to Lemma \ref{ABg}, we can take an affine cover of $\p^2\times\p^2-\Delta(\p^2)$ which trivializes $\widetilde{\mc_2^R}$.
\end{proof}  
  
Denote by $0$ the only singular point in $C^R$.  $C^R-\{0\}=\mathbb{A}^1\sqcup\mathbb{A}^1$.  $\widehat{\mo}_{C^R,0}\simeq \ts/(xy)=\bc[[x,y]]/(xy)$.  We make a table for ideals of $\ts/(xy)$ as Table II.

\begin{table}[h]
\centering
\begin{tabular}{|c|c|c|}
\hline\multicolumn{3}{|l|}{\textbf{Table II}\qquad\qquad\qquad\qquad Ideals $I$ of $\ts$ containing $(xy)$}\\
\hline Co-length of $I$ & Ideal $I$ & $I\cap (\mathfrak{m}^2-\mathfrak{m}^3)$\\
\hline 1 & $\mathfrak{m}$ & $\bc x^2\oplus \bc xy\oplus\bc y^2$\\ \hline
2 & $\mathfrak{m}^2+(kx+k'y)\ts, (k,k')\neq0$ &$\bc x^2\oplus \bc xy\oplus\bc y^2$\\ \hline
\multirow{3}{*}{3} & $\km^2$ &  $\bc x^2\oplus \bc xy\oplus\bc y^2$\\ \cline{2-3}
&$\km^3+(x+ky^2)\mathtt{S}$ & $\bc x^2\oplus \bc xy$ \\ \cline{2-3}
&$\km^3+(y+kx^2)\mathtt{S}$ & $\bc xy\oplus\bc y^2$\\ \hline 
\multirow{3}{*}{4}& $xy\ts+(kx^2+k'y^2)\ts+\km^3, (k,k')\neq0$ & $\bc xy\oplus \bc(kx^2+k'y^2)$\\ \cline{2-3}
&$(x+ky^3)\ts+\km^4$& $\bc x^2\oplus \bc xy$\\ \cline{2-3}
&$(y+kx^3)\ts+\km^4$& $\bc xy\oplus \bc y^2$\\ \hline
\multirow{7}{*}{5} &$xy\ts+\km^3$&$ \bc xy$\\ \cline{2-3}
&$xy\ts+(x^2+ky^3)\ts+\km^4, k\neq 0$&$\bc xy$\\ \cline{2-3}
&$xy\ts+x^2\ts+\km^4$&$\bc x^2\oplus\bc xy$\\ \cline{2-3}
&$xy\ts+(y^2+kx^3)\ts+\km^4, k\neq 0$&$\bc xy$\\ \cline{2-3}
&$xy\ts+y^2\ts+\km^4$&$\bc xy\oplus\bc y^2$\\ \cline{2-3}
&$(x+ky^4)\ts+\km^5$&$\bc x^2\oplus \bc xy$\\ \cline{2-3}
&$(y+kx^4)\ts+\km^5$&$\bc xy\oplus\bc y^2$\\ \hline
\multirow{7}{*}{6}&$xy\ts+(kx^3+k'y^3)\ts+\km^4,(k,k')\neq0$&$ \bc xy$\\ \cline{2-3}
&$xy\ts+(x^2+ky^4)\ts+\km^5, k\neq 0$&$\bc xy$\\ \cline{2-3}
&$xy\ts+x^2\ts+\km^5$&$\bc x^2\oplus\bc xy$\\ \cline{2-3}
&$xy\ts+(y^2+kx^4)\ts+\km^5, k\neq 0$&$\bc xy$\\ \cline{2-3}
&$xy\ts+y^2\ts+\km^5$&$\bc xy\oplus\bc y^2$\\ \cline{2-3}
&$(x+ky^5)\ts+\km^6$&$\bc x^2\oplus \bc xy$\\ \cline{2-3}
&$(y+kx^5)\ts+\km^6$&$\bc xy\oplus\bc y^2$\\ \hline
\end{tabular}
\end{table}
$\sigma_2$ acts on $Hilb^{[6]}(C^R)$ by exchanging the two irreducible components of $C^R$.  Write $Hilb^{[6]}(C^R)=H^x\sqcup H^s\sqcup H^y$ such that $\sigma_2(H^x)=H^y$ and $\sigma_2(H^s)=H^s$.  Let $H^x_n=\mr_n\cap H^x$ and analogously we have $H^y_n$ and $H_n^s$.  Then $\sigma_2(H^x_n)=H^y_n,\sigma_2(H^s_n)=H_n^s.$

\begin{lemma}\label{ABq}\begin{enumerate}\item $[H_0^x]=[\mathbb{A}^6+2\mathbb{A}^5+3\mathbb{A}^4+3\mathbb{A}^3+2\mathbb{A}^2-1];$
\item $[H^x_1]=[\mathbb{A}^6+2\mathbb{A}^5+2\mathbb{A}^4+2\mathbb{A}^3+2\mathbb{A}^2+2\mathbb{A}^1];$
\item $[H^x_2]=[\p^6];$
\item $[H^s_0]=[\mathbb{A}^6+\mathbb{A}^4\times\p^1+\mathbb{A}^2\times\p^1+\p^1];$ 
\item $H^s_1=H^s_2=\emptyset.$
\end{enumerate}
\end{lemma}
\begin{proof}$\forall I_6\in Hilb^{[6]}(C^R)$, $\exists~0\leq i\leq 6$, such that $I_6=I_i^{0}\cap I_{6-i}$ with $[I^0_i]\in Hilb^{i}(\{0\})$ and $[I_{6-i}]\in Hilb^{[6-i]}(C^R-\{0\})=Hilb^{[6-i]}(\mathbb{A}^1\sqcup\mathbb{A}^1)$.  Then the lemma can be proved by elementary analysis and computation case by case.
\end{proof}
\begin{lemma}\label{ABh}$[\mr^R_n]=[H_n^x\times (\p^2\times\p^2-\Delta(\p^2))]+[H^s_n\times Sym^2(\p^2)-\p^2].$
\end{lemma}
\begin{proof}$\sigma_2(H_n^x\times (\p^2\times\p^2-\Delta(\p^2))=H_n^y\times (\p^2\times\p^2-\Delta(\p^2))$ and $\sigma_2(H_n^s\times (\p^2\times\p^2-\Delta(\p^2))=H_n^s\times (\p^2\times\p^2-\Delta(\p^2))$.  Hence the lemma.
\end{proof}

Now combine Lemma \ref{ABic}, Lemma \ref{ABg}, Lemma \ref{xsid}, Lemma \ref{ABr}, Lemma \ref{ABq} and Lemma \ref{ABh}, we get $[\mr_n]$ for $n=0,1,2.$  Since $[\Omega_2^{[6]}]=\sum_{n=0}^2 [\ms_n]$ and $[\ms_n\times\p^n]=[\mr_n]$, Lemma \ref{ctome} follows after direct computation.
\end{proof}


\end{document}